\documentclass{amsart}
\usepackage{amsthm}
\usepackage{amsfonts}
\usepackage{amsmath}
\usepackage{color}
\usepackage{enumerate}
\usepackage{mathrsfs}
\usepackage{graphicx}

\usepackage[symbol]{footmisc}

\newtheorem{theorem}{Theorem}[section]
\newtheorem{proposition}[theorem]{Proposition}
\newtheorem{lemma}[theorem]{Lemma}
\newtheorem{corollary}[theorem]{Corollary}
\newtheorem{definition}[theorem]{Definition}
\newtheorem{hypothesis}[theorem]{Hypothesis}

\theoremstyle{remark}
\newtheorem{remark}[theorem]{Remark}

\numberwithin{equation}{section}

\newcommand{\R}{{\mathbb R}}

\newcommand{\C}{{\mathbb C}}

\newcommand{\bbC}{{\mathbb{C}}}

\newcommand{\bbN}{{\mathbb{N}}}

\newcommand{\bbR}{{\mathbb{R}}}

\newcommand{\bbZ}{{\mathbb{Z}}}

\newcommand{\bsA}{{\boldsymbol{A}}}
\newcommand{\bsB}{{\boldsymbol{B}}}
\newcommand{\bsC}{{\boldsymbol{C}}}
\newcommand{\bsD}{{\boldsymbol{D}}}

\newcommand{\bsH}{{\boldsymbol{H}}}
\newcommand{\bsI}{{\boldsymbol{I}}}

\newcommand{\bsL}{{\boldsymbol{L}}}

\newcommand{\bsP}{{\boldsymbol{P}}}

\newcommand{\bsT}{{\boldsymbol{T}}}

\newcommand{\cB}{{\mathcal B}}

\newcommand{\cD}{{\mathcal D}}

\newcommand{\cH}{{\mathcal H}}

\newcommand{\cL}{{\mathcal L}}

\newcommand{\Lxi}{\xi_L}

\newcommand{\Lf}{f_L}

\DeclareMathOperator{\const}{const}
\DeclareMathOperator{\iindex}{index}

\DeclareMathOperator{\ran}{ran}
\DeclareMathOperator{\dom}{dom}
\DeclareMathOperator{\tr}{tr}

\DeclareMathOperator{\erf}{erf}
\DeclareMathOperator{\sfl}{sf}

\DeclareMathOperator*{\slim}{s-lim}

\renewcommand{\Im}{\text{\rm Im}}

\newcommand{\loc}{\operatorname{loc}}

\newcommand{\ind}{\operatorname{index}}
\newcommand{\no}{\notag}
\newcommand{\lb}{\label}
\newcommand{\f}{\frac}

\begin{document}

\title[Witten index and spectral shift function]
{The Witten index and the spectral shift function}

\author[A.\ Carey]{Alan Carey}  
\address{Mathematical Sciences Institute, Australian National University, 
Kingsley St., Canberra, ACT 0200, Australia
and School of Mathematics and Applied Statistics, University of Wollongong, NSW, Australia,  2522}  
\email{alan.carey@anu.edu.au}  
%

\author[G.\ Levitina]{Galina Levitina} 
\address{Mathematical Sciences Institute, Australian National University, 
Kingsley St., Canberra, ACT 0200, Australia}  
\email{galina.levitina@anu.edu.au}  

\author[D.\ Potapov]{Denis Potapov} 
\address{School of Mathematics and Statistics, UNSW, Kensington, NSW 2052,
Australia} 
\email{d.potapov@unsw.edu.au}

\author[F.\ Sukochev]{Fedor Sukochev}
\address{School of Mathematics and Statistics, UNSW, Kensington, NSW 2052,
Australia} 
\email{f.sukochev@unsw.edu.au}

\begin{abstract}
In \cite{APSIII} Atiyah, Patodi and Singer introduced spectral flow for elliptic operators on odd dimensional compact manifolds.
They argued that it could be computed from the Fredholm index of an elliptic operator on a manifold of one higher dimension.
A general proof of this fact was produced by Robbin-Salamon \cite{RS95}. 
In \cite{GLMST}, a start was made on extending these ideas to operators with some essential spectrum as occurs on non-compact manifolds.
The new ingredient introduced there was to exploit scattering theory following the fundamental paper \cite{Pu08}.
These results do not apply to differential operators directly, only to pseudo-differential operators on manifolds,
due to the restrictive assumption that spectral flow is considered between an operator and {its perturbation by a relatively trace-class operator}. 
In this paper  we extend the main results of these earlier papers to spectral flow between an operator and a perturbation satisfying a higher $p^{th}$ Schatten class condition for $0\leq p<\infty$,
thus allowing differential operators on manifolds of any dimension $d<p+1$. In fact our main result does not assume any ellipticity or Fredholm properties
at all and proves an operator theoretic trace formula motivated by \cite{BCPRSW, CGK16}. We illustrate our results 
using Dirac type operators on $L^2(\bbR^d)$ for arbitrary $d\in\bbN$ (see Section \ref{ch_examples}). In this setting Theorem~\ref{thm_principla trace formula} substantially extends \cite[Theorem 3.5]{CGGLPSZ16}, where the case $d=1$ was treated.

\end{abstract}

\maketitle

%
 
\section{Introduction} 
\subsection{Motivation and background}

Suppose that $\{A(t)\}_{t\in\bbR}$ is a family of (possibly unbounded) self-adjoint operators in a separable Hilbert space $\cH$ and consider the operator 
$$\bsD_\bsA=\frac{d}{dt}+A(t),$$
in the Hilbert space $L_2(\bbR,\cH)$.
Operators of this form were studied by  Atiyah, Patodi and  Singer \cite{APS75I}, \cite{APS75II}, \cite{APSIII} with $A(t),$ \, $t\in\bbR,$ a first order elliptic differential operator on a compact odd-dimensional manifold with the asymptotes (in a suitable topology) 
$A_{\pm}=\lim_{t\to\pm\infty}A(t)$
boundedly invertible and purely discrete spectra of $A_\pm, A(t), t\in\bbR$. The assumption that $A_\pm$ are boundedly invertible guarantees that the operator $\bsD_\bsA$ is Fredholm, and therefore the Fredholm index, $\ind(\bsD_\bsA)$, of the operator $\bsD_\bsA$ is well-defined. Atiyah, Patodi and  Singer (APS) showed that $\ind(\bsD_\bsA)$ is equal to the spectral flow $\sfl\{A(t)\}_{t=-\infty}^\infty$ of the path $\{A(t)\}_{t\in\bbR}$. The spectral flow in APS is intuitively understood as the number of  eigenvalues
(counting multiplicities) of $A(t)$ that go from negative to positive minus the number going from positive to negative as $t$ runs from $-\infty$ to $+\infty$. 

Analogous theorems were later studied by other  methods (see for example \cite{BBWo93}) culminating in a definitive treatment for certain self-adjoint differential operators with compact resolvent in a paper of Robbin--Salamon, \cite{RS95}. Their paper has implications for  many applications including those to Morse theory, Floer
homology, Morse and Maslov indices, Cauchy-Riemann operators, and all the way to oscillation
theory. 
%

%
With the exception of \cite{BCPRSW}, the focus in mathematics on the equality of the index of the operator $\bsD_\bsA$ and the spectral flow has been mainly on geometrically defined operators associated to compact manifolds.
Physicists, however, are interested in the case of non-compact manifolds and non-geometric examples. The stumbling block for index formulas in that situation
is the presence of essential spectrum. Motivated by ideas from scattering theory  an approach to a Robbin-Salamon type result for paths of self-adjoint Fredholm operators with
some essential spectrum
was initiated in \cite{Pu08} and  \cite{GLMST}. However, the key assumption in \cite{GLMST} is that  spectral flow 
is only considered between self-adjoint operators that differ by a so-called relatively trace class perturbation.  This latter assumption is satisfied by certain pseudo differential perturbations
of a fixed differential operator
but does not apply to paths of differential operators even in one dimension. As a result, the promising start made in \cite{GLMST}
to generalising the Robbin-Salamon theorem (so that it applied to operators with some essential spectrum as occurs in the non-compact manifold case), ran into difficulties.

%
%
%
Relevant to the issue of allowing essential spectrum is that some time ago
 Witten \cite{Witten}  gave a proposal for extending index theory beyond the Fredholm setting, particularly motivated by supersymmetric quantum theory. 
For certain Dirac type operators (with essential spectrum) Witten's ideas could be shown to produce index theorems \cite{BGGSS87}, \cite{GS88}.
Some of this early work also involves non-Fredholm operators 
and it was revisited  in \cite{CGPST_Witten} and related to Pushnitski's work under the same relative trace class assumption as in \cite{GLMST}.  
This led some of the present authors in \cite{CGLS16} to begin to complete the program started in \cite{GLMST}.
The new ingredient in \cite{CGLS16} is an approximation technique
that enabled the main theorems in \cite{GLMST} to be extended.
 It is this approximation technique that underpins the advances described in this paper.

We make four main advances. First, our results are purely operator theoretic and hence apply in non-geometrical settings. 
Second in the setting of the Robbin-Salamon theorem we show that when one drops the assumption of invertible endpoints for the path
one still obtains an index formula except that the Witten index replaces the Fredholm index.  Third, we allow essential spectrum in the case of paths of self-adjoint Fredholm 
operators.  Fourth, we show that we obtain information even in the case where the path of self-adjoint operators is not Fredholm.  It is in this final situation that 
substantial results from quantum mechanical scattering theory are essential. 

\subsection{An overview of our results}

We now briefly discuss the setting and results of the present paper.
For the purpose of an accessible introduction we consider here a simplified situation of the general setting under which we prove the results. 
%

We start with a self-adjoint unbounded operator  $A_-$ densely defined on a separable complex Hilbert space $\cH$ and suppose that $B$ is a self-adjoint bounded perturbation of $A_-$. If the perturbation $B$ is a relative trace-class perturbation of $A_-$, that is, $B(A_-+i)^{-1}$ is a trace-class operator, then the main assumption in \cite{GLMST}, \cite{CGPST_Witten} is satisfied. Hence, the results summarised in the previous subsection are known to be true. However, the critical fact for partial differential operators  in general  is that perturbations by lower order operators
satisfy relative Schatten--von Neumann class constraints but not relative trace class constraints.
To describe this, suppose for example that we consider $A_-$ to be 
the flat space Dirac operator acting in $L^2(\bbR^n) \otimes \bbC^m$ and the perturbation $B$ to be given by the multiplication operator by a smooth, $m \times m$ matrix-valued bounded function 
$F:\bbR^n\to M^{m \times m}(L^{\infty}(\bbR) \cap C^{\infty}(\bbR))$, $m \in \bbN$.
Under suitable decay conditions at infinity for $F$, the product $B(1+A_-^2)^{-s/2}$ is trace class for $s>n$ and for no smaller value of $s$ (see \cite[Remark~4.3]{Simon_book}). Therefore, the case where we may set $s=1$ and obtain a relative trace class perturbation  requires we work in dimension zero. The only way that we know of to obtain
a relative trace class perturbation in $n$-dimensions is to replace $F$ by certain pseudo-differential operators.  It eventuates, as we explain in later sections,
that an appropriate choice of pseudo-differential perturbations can be used to approximate the kind of perturbation of Dirac operators that arise  in examples.
In this way the results of \cite{GLMST} and \cite{CGPST_Witten} can be brought to bear on the higher dimensional situation.

Thus, if we want to include multidimensional examples in our setting the assumption on the perturbation $B$ should be that the operator $B(A_-+i)^{-p-1}$ is a trace-class operator for some fixed $p\in\bbN\cup\{0\}$. This assumption includes, in particular, the assumptions of \cite{Pu08} and \cite{GLMST}, \cite{CGPST_Witten}. However, what is more important, in contrast to these three papers,  our  assumption is satisfied by 
Dirac type operators on certain non-compact manifolds (see \cite{CGRS_memoirs}). See Section \ref{ch_examples} below for the example of the Dirac operator on $\bbR^d, d\in\bbN$. 

To discuss the connection to index theory, we introduce now the `suspension' operator $\bsD_\bsA$ as in \cite{APSIII}, \cite{RS95}.
Let $\theta:\mathbb R\to \mathbb R$ be a smooth function (with integrable positive derivative) interpolating between zero and one
in the sense that $\lim_{t\to-\infty}\theta(t)=0$ and $\lim_{t\to\infty}\theta(t)  =1$.
Denoting by $M_\theta$ the operator on $L^2(\bbR)$ of multiplication by $\theta$, introduce the operator $\bsD_\bsA^{}$ in the Hilbert space $L^2(\bbR)\otimes \cH$ by 
\begin{equation*}
\bsD_\bsA^{} =\f{d}{dt}\otimes 1+ 1\otimes A_- +M_\theta\otimes B.
\end{equation*} 
Here the operator $d/dt$ in $L^2(\bbR)$  is the differentiation operator with domain being the Sobolev space $W^{1,2} \big(\bbR)$,
so that $\bsD_\bsA^{}$ is defined on $W^{1,2} \big(\bbR)\otimes \dom(A_-)$. For ease of notation we will usually identify 
$L^2(\bbR)\otimes \cH$ and $L^2(\bbR;\cH)$.

In order to relate the index of the operator $\bsD_\bsA$ with the spectral flow (when defined) of the family $\{A_-+\theta(t)B\}_{t\in\bbR}$ we establish the first main result of this current paper, the {\it principal trace formula}. Namely, under some additional mild assumptions for any $t>0$ we prove the following relation (see Theorem \ref{thm_principla trace formula} below):
\begin{align}\label{ch_principla trace formula_intro_formula}
\tr\Big(e^{-t\bsD_{\bsA}^{} \bsD_{\bsA}^{*}}-e^{-t\bsD_{\bsA}^{*} \bsD_{\bsA}^{}}\Big)=-\Big(\frac{t}{\pi}\Big)^{1/2}\int_0^1\tr\Big(e^{-tA_s^2}B\Big)ds,\\
 \quad A_s=A_-+sB, \quad s\in[0,1],\nonumber
\end{align}
noting that our hypotheses guarantee both sides of the relation are well-defined. Here, $\tr$ denotes the classical trace on the algebra $\cB(\cH)$ of all bounded linear operator on a Hilbert space $\cH$.

The key point to note is that this is an operator identity that makes no assumptions on the spectrum of the operators on either
side of the relation. We investigate one consequence of this fact using scattering theory methods and involving the so-called Witten index. 
There are other applications and examples that require considerable preliminary details to explain and we defer those to a later date.

Note that if we impose the assumption that the operators $A_-$ and $A_+:=A_-+B$ have discrete spectrum and are unitarily equivalent and invertible then the operators $A_\pm$ and $\bsD_{\bsA}$ are Fredholm and 
in that case the left-hand side of \eqref{ch_principla trace formula_intro_formula} is the Fredholm index of  $\bsD_{\bsA}$ \cite{GS88} while the right hand side is the spectral flow along the path $\{A_-+\theta(t)B\}_{t\in\bbR}$ \cite{ACS, CP2} (note that we define spectral flow analytically in this paper as in \cite{CP2} rather than use the APS point of view) . Thus the principal
trace formula entails a version of the Robbin-Salamon theorem.  

However, if we assume that the endpoints $A_\pm$ are not invertible, the principal trace formula remains true, but the left-hand side is no longer the Fredholm index
since the operator $\bsD_\bsA^{}$ is no longer Fredholm. However the right-hand side will still be spectral flow if $A_\pm$ have discrete spectrum.
As the right-hand side is independent of $t$ in this case \cite{CP2} so too is the left-hand side and as discussed below, it is the Witten index of $\bsD_\bsA^{}$.

If the operator $A_-$ has some essential spectrum then neither the left hand side nor the right hand side of the principal trace formula can be proved to be independent of $t$. There are then two possible asymptotic quantities that we might consider.  The first, the limit as $t\to 0$ on the left-hand side of the principal trace formula,
 gives what has been referred to in the physics literature as the anomaly.  In the context of the Atiyah-Singer index theorem it gives the local form of the 
theorem involving integrals of characteristic classes. There is evidence \cite{CGK15} that for Dirac type operators with some essential spectrum the anomaly
can be expressed in terms of a local formula of the same type as arises in the Atiyah-Singer theorem. 

Our main interest in the present paper lies in the second limit as $t\to\infty$  of \eqref{ch_principla trace formula_intro_formula}. On the left hand side this limit, if it exists, has been termed the Witten index \cite{GS88}. 
The study of the limit $t\to\infty$ on the right-hand side of the principal trace formula is made possible by exploiting the Krein spectral shift function from quantum mechanical scattering theory. As an application of the principal trace formula, we obtain 
a generalisation of the original formulas of Pushnitski
relating the spectral shift function for the pair $(A_+, A_-)$ with that for the pair $(\bsD_\bsA^{\ast}\bsD_\bsA^{},\bsD_\bsA^{}\bsD_\bsA^{\ast})$, as well as formula of Witten index of $\bsD_\bsA$ in terms of spectral shift function for the pair  $(A_+, A_-)$ (see Theorem \ref{thm_WI=ssf}).

\subsection{Summary of the exposition}

In Section \ref{ch_prelim} we collect some preliminaries.  
In Section \ref{ch_approximation} we give a detailed exposition of the setting and explain the approximation scheme that we use in order to apply the results of \cite{Pu08,GLMST}.
The idea is to use a `cut-off' that, in the case of differential operators,  produces a sequence of pseudo-differential approximants.
This is the technical heart of the paper.
The point is that for the approximants the principal trace formula is quickly established. Then in Section \ref{ch_principla trace formula} we have to handle the convergence of
both sides as we remove our cut-off thus proving the principal trace formula. 
Finally, in  Section \ref{ch_WI} we firstly  introduce the spectral shift function adapted to our setting. There is a subtle point that we need to be careful about in that for the pair
$A_\pm$ the spectral shift function is a priori only defined up to a constant.  Fixing this constant is essential but is not a trivial matter. With these preliminaries
out of the way we move to our analogue of Pushnitski's formula. Then we discuss the Witten index.


In the final Section \ref{ch_examples} we show that the our hypothesis allows consideration of  Dirac operators in arbitrary dimensions.  

{

\subsection*{Acknowledgements}  
A.C., D.P., and F.S.\ thank the Erwin Schr\"odinger International 
Institute for Mathematical Physics (ESI), Vienna, Austria, for funding support 
for this collaboration in the form of a Research-in Teams project, `Scattering Theory 
and Non-Commutative Analysis' from June 22 to July 22, 2014. 
G.L.\ is indebted to Gerald Teschl for a kind invitation to visit the 
University of Vienna, Austria, for a period of two weeks in June/July of 2014. A.C., G.L., F.S.
thank BIRS for funding a focussed research group on the topic of this paper in June 2017.
A.C., G.L., and F.S.\ gratefully acknowledge financial support from the Australian Research 
Council.  A.C. thanks the Humboldt Stiftung for support. G.L. acknowledges the support of Australian government
scholarships.
We thank Fritz Gesztesy, Jens Kaad, Harald Grosse and Dmitriy Zanin for many conversations on the subject matter of this
article.
}

\section{Preliminaries}\label{ch_prelim}

In the present section we introduce the class of so-called $p$-relative trace class perturbations of a given self-adjoint operator and collect the preliminary properties of this class of perturbations. This class of perturbations constitutes the essential part of the main hypothesis for our results (see Hypothesis \ref{hyp_final}).

In what follows $\cH$ denotes a separable complex Hilbert space. The algebra of all linear bounded operators on $\cH$ is denoted by $\cB(\cH)$ and is equipped with the uniform norm $\|\cdot\|$. 
%
%
The corresponding $\ell^p$-based Schatten--von Neumann ideals on $\cH$ are denoted by $\cL_p (\cH)$, with associated norm abbreviated by 
$\|\cdot\|_p$, $p \geq 1$. Moreover, 
${\tr(A)}$ denotes 
the trace of a trace class operator $A\in\cL_1(\cH)$. 

\begin{definition}Let $A_0$ be a self-adjoint (possibly unbounded) operator on  $\cH$ and let $p\in \bbN\cup\{0\}$. A bounded self-adjoint operator $B\in\cB(\cH)$ is called a $p$-relative trace class perturbation (with respect to $A_0$) if 
\begin{equation}\label{def_p_relative_inc}
B(A_0+i)^{-p-1}\in\cL_1(\cH).
\end{equation}
In what follows, we choose the smallest $p$, such that \eqref{def_p_relative_inc} holds.
\end{definition}

It is clear that one can take the resolvent parameter for the operator $A_0$ to be any point $z\in\bbC\setminus \bbR$ for the definition of $p$-relative trace class perturbations. In addition, since $B$ is self-adjoint, inclusion \eqref{def_p_relative_inc} is equivalent to the inclusion $(A_0+i)^{-p-1}B\in\cL_1(\cH)$. Furthermore, elementary functional calculus implies that for any $C\in\cB(\cH)$ the operators 
$C(A_0+i)^{-q}$ and $C(A_0^2+1)^{-q/2}$
belong to the same ideal of $\cB(\cH)$. In what follows we use this fact repeatedly without additional explanations.
%

Classical interpolation theory for Schatten ideals (see see e.g.  \cite[Theorem 2.9]{Simon_book}, \cite[Section III, Theorem 13.1]{GK}) implies the following simple result.
\begin{lemma}\label{lem_B_+_relative_dif_powers}
Suppose that $B$ is a $p$-relative trace-class perturbation of $A_0$. Then 
\begin{enumerate}
\item For any $j=1,\dots, p+1$, we have 
\begin{equation}\label{B_+relative_dif_powers}
B(A_0+i)^{-j}\in\mathcal{L}_{\frac{p+1}{j}}(\cH),\quad j=1,\dots, p+1,
\end{equation}
and 
\begin{equation*}
\|B(A_0+i)^{-j}\|_{\frac{p+1}{j}}\leq \|B\|^{\frac{j}{p+1}}\cdot\|B(A_0+i)^{-p+1}\|_1^{\frac{p+1-j}{p+1}}.
\end{equation*}

\item For any $j=1,\dots, p+1$ and any $k,l\in\bbN\cup\{0\}$ such that $k+l=j,$ we have 
$(A_0+i)^{-k}B(A_0+i)^{-l}\in\mathcal{L}_{\frac{p+1}{j}}(\cH).$
and 
\begin{equation*}
\|(A_0+i)^{-k}B(A_0+i)^{-l}\|_{\frac{p+1}{j}}\leq \|B(A_0+i)^{-j}\|_{\frac{p+1}{j}}.
\end{equation*}
\end{enumerate}
\end{lemma}
%
%

Inclusion \eqref{B_+relative_dif_powers} combined with Weyl's theorem (see e.g. \cite[Theorem 9.13]{Weidmann}) implies the stability of essential spectra: 
$\sigma_{\text{ess}}(A_0+B)=\sigma_{\text{ess}}(A_0).$

Next, we shall show that the assumption that $B$ is a $p$-relative trace class perturbation of a self-adjoint operator $A_0$, guarantees that
$(A_0+B-ai)^{-p}-(A_0-ai)^{-p}$ is a trace-class operator. In addition, we establish an approximation result for the difference $(A_0+B-ai)^{-p}-(A_0-ai)^{-p}$ necessary for the proof of the principal trace formula. 

Firstly, we recall a result which is used repeatedly throughout the paper. This result can be found in \cite{Simon_book} and \cite[Lemma 3.4]{GLMST}.
\begin{lemma}\label{so_inLp}
Let $p\in[1,\infty)$ and assume that $R,R_n,T,T_n\in\cB(\cH)$, 
$n\in\bbN$, are such that $R_n\to R$ and $T_n\to T$ in the strong operator topology, and that
$S,S_n\in\cL_p(\cH)$, $n\in\bbN$, satisfy 
$\lim_{n\to\infty}\|S_n-S\|_p=0$.
Then $\lim_{n\to\infty}\|R_n S_n T_n^\ast - R S T^\ast\|_p=0$.
\end{lemma}

We now introduce special spectral `cut-off'  approximants $B_n$, $n\in\bbN$, for a $p$-relative trace-class perturbation $B$, by setting 
$P_n=\chi_{[-n,n]}(A_0)$
and 
$B_{n}:=P_nBP_n.$
Throughout this section we will assume that $P_n$ and $B_n$, $n\in\bbN$,  are as in the above definition.

It follows from spectral theory that $
\slim_{n\to\infty} P_n=1,$
where $\slim$ denote convergence with respect to the strong operator topology. 
We also note that the equality \eqref{def_p_relative_inc} together with the definition of the projections $P_n$ implies that
\begin{equation}\label{B_+,n_trace_class}
B_{n}=P_nBP_n\in \cL_1(\cH).
\end{equation}

\begin{remark}\label{rem_exact_P_n_immaterial}
The precise form of the cut-offs $P_n$ is of course immaterial. We just need several facts: that $\slim_{n\to\infty}P_n=1$, $\sup_{n\in\bbN}\|P_n\|<\infty$ and that $P_nBP_n\in \cL_1(\cH)$. $\diamond$
\end{remark}

The following lemma gathers some simple properties of the approximants $A_0+B_n$, $n\in\bbN$. 

\begin{lemma}
\label{approx_A_n}With $B$, $B_n$ and $P_n$ be as above  we have that
\begin{enumerate}
\item $A_0+B_n\rightarrow A_0+B$ in the strong resolvent sense as $n\to\infty$;
\item Let $j\in\bbN$. For any $k,l\in\bbN$ such that $k+l\geq j$ and $z\in\bbC\setminus\bbR$, we have 
$$\lim_{n\rightarrow \infty}\Big\| (A_0-z)^{-k}B_{n}(A_0-z)^{-l}-(A_0-z)^{-k}B(A_0-z)^{-l}\Big\|_{\frac{p+1}{j}}=0.$$
\end{enumerate}
\end{lemma}
\begin{proof}
(i). Since the operator $B$ is bounded, it follows that the operators $A_0+B_n$ and $A_0+B$ have common core $\dom(A_0)$. Therefore, by \cite[Theorem VIII.25 (a)]{RS_book_I} it is sufficient to show that $(A_0+B_n)\xi\to (A_0+B)\xi$ for all $\xi\in\dom(A_0)$. This easily follows from the fact that $B_n\to B$ in the strong operator topology.

(ii). Since $k+l\geq j$, Lemma \ref{lem_B_+_relative_dif_powers} (ii) implies that 
$(A_0-z)^{-k}B(A_0-z)^{-l}\in\cL_{\frac{p+1}{j}}(\cH).$
Therefore, since  
$ (A_0-z)^{-k}B_{n}(A_0-z)^{-l}= P_n(A_0-z)^{-k}B(A_0-z)^{-l}P_n,$
and $P_n\to 1$ in the strong operator topology, the assertion follows from Lemma \ref{so_inLp}.

\end{proof}

Next,  fix $j=1,\dots, p$ and consider the difference
$(A_0+B-z)^{-j}-(A_0-z)^{-j}.$
Our aim now is to show that this operator belongs to a Schatten ideal $\cL_q(\cH)$, with $q\geq 1$ chosen as small as possible. 
To lighten notation, we introduce
$A_1=A_0+B.$

Using the elementary identity
\begin{equation*}
C_1^j - C_2^j= \sum_{k_0+k_1=j-1}C_1^{k_0} [C_1 - C_2] C_2^{k_1}, \quad C_1,C_2 \in \cB(\cH), 
\; j \in \bbN,
\end{equation*}
and the resolvent identity we can write
\begin{align*}
\begin{split}
\left( A_1-z \right)^{-j}& - \left( A_0 -z \right)^{-j}
\\
&=\sum_{k_0+k_1=j-1}(A_1-z)^{-k_0}\Big(\left( A_1 -z \right)^{-1} - \left( A_0 -z \right)^{-1}\Big)(A_0-z)^{-k_1}
\\
&=-\sum_{k_0+k_1=j-1}(A_1-z)^{-k_0-1}B(A_0-z)^{k_1-1}.
\end{split}
\end{align*}

Writing 
$$(A_1-z)^{-k_0-1}=(A_0-z)^{-k_0-1}+\big((A_1-z)^{-k_0-1}-(A_0-z)^{-k_0-1}\big)$$ and repeating the same argument for the second term on the right-hand side we obtain
\begin{align*}
\begin{split}
\left( A_1 -z \right)^{-j}& - \left( A_0 -z \right)^{-j}
\\
&=-\sum_{k_0+k_1=j-1}(A_0-z)^{-k_0-1}B(A_0-z)^{k_1-1}\\
&\qquad-\sum_{k_0+k_1=j-1}\Big((A_1-z)^{-k_0-1}-(A_0-z)^{-k_0-1}\Big)B(A_0-z)^{k_1-1}\\
&=-\sum_{k_0+k_1=j-1}(A_0-z)^{-k_0-1}B(A_0-z)^{k_1-1}\\
&\qquad -\sum_{k_0+k_1+k_2=j-1}(A_1-z)^{-k_0-1}B(A_0-z)^{-k_1-1}B(A_0-z)^{k_1-1}.\end{split}
\end{align*}

If, for arbitrary  $j=1,\dots, p$, $i\in\bbN$, $z\in\bbC\setminus\bbR$ and $X_1,\dots, X_i\in\cB(\cH)$, we introduce the following operators
\begin{align}\label{T_n_resolvents}
\begin{split}
T_1^{(j)}&(X_1)=\sum_{k_0+k_1=j-1}(A_0-z)^{-k_0-1}X_1(A_0-z)^{-k_1-1},\\
T_2^{(j)}&(X_1,X_2)=\sum_{k_0+k_1+k_2=j-1}(A_0-z)^{-k_0-1}X_1(A_0-z)^{-k_1-1}X_2(A_0-z)^{-k_2-1},\\
&\dots\\
T_i^{(j)}&(X_1,\dots, X_i)\\
&=\sum_{k_0+\dots+k_i=j-1}(A_0-z)^{-k_0-1}X_1(A_0-z)^{-k_1-1}\dots X_i(A_0-z)^{-k_i-1},
\end{split}
\end{align}
and \begin{align*}
R_i^{(j)}&(B; X_1,\dots, X_i)\\
&=\sum_{k_0+\dots+k_i=j-1}(A_1-z)^{-k_0-1}X_1(A_0-z)^{-k_1-1}\dots X_i(A_0-z)^{-k_i-1},
\end{align*}
then we have a Taylor formula for the difference $(A_1-z)^{-j}-(A_0-z)^{-j}$ of the form,
\begin{align}\label{diff_resolvent_Ti}
\begin{split}
\left( A_1 -z \right)^{-j}& - \left( A_0 -z \right)^{-j}
\\
&=T_1^{(j)}(B)+\dots +T_j^{(j)}(B,\dots, B)+R_{j+1}^{(j)}(B;B,\dots, B).
\end{split}
\end{align}

\begin{lemma}\label{mappings_Ti_lem}
Fix $j=1,\dots, p$ and assume that $B$ is a $p$-relative trace-class perturbation of $A_0$. Set $P_n=\chi_{[-n,n]}(A_0)$ and  $B_n:=P_nBP_n$. Using notations of \eqref{T_n_resolvents} and \eqref{diff_resolvent_Ti}, the following assertions hold:
\begin{enumerate}
\item For every $i\in\bbN$ we have $T_i^{(j)}(B,\dots, B)\in \cL_{\frac{p+1}{j+i}}(\cH)$.
\item For every $i\in\bbN$ we have the convergence
$$\lim_{n\to\infty}\Big\|T_i^{(j)}(B_n,\dots, B_n)-T_i^{(j)}(B,\dots, B)\Big\|_{\frac{p+1}{j+i}}=0.$$
\item For every $i\in\bbN$ we have $R_i^{(j)}(B; B,\dots, B)\in \cL_{\frac{p+1}{i}}(\cH)$.
\item For every $i\in\bbN$ we have the convergence
$$\lim_{n\to\infty}\Big\|R_i^{(j)}(B_n;B_n,\dots, B_n)-R_i^{(j)}(B;B,\dots, B)\Big\|_{\frac{p+1}{i}}=0.$$
\end{enumerate}
\end{lemma}

\begin{proof}
(i). Since the operator $B$ is a $p$-relative trace-class perturbation of $A_0$, Lemma~\ref{lem_B_+_relative_dif_powers} implies that 
\begin{equation}\label{Ti_1}
(A_0-z)^{-k_0-1}B(A_0-z)^{-k_1-1}\in \cL_{\frac{p+1}{k_0+k_1+2}}(\cH)
\end{equation}
 and 
 \begin{equation}\label{Ti_2}
 B(A_0-z)^{-k_l-1}\in\cL_{\frac{p+1}{k_l+1}},\quad l=2,\dots i.
 \end{equation}

Hence, by the  H\"older inequality we have that 
\begin{align*}(A_0-z)^{-k_0-1}B(A_0-z)^{-k_1-1}\dots B(A_0-z)^{-k_i-1}\in \cL_{\frac{p+1}{k_0+k_1+\dots+k_i+i+1}}(\cH).
\end{align*} 
Since $k_0+k_1+\dots+k_i=j-1$, we then obtain the required result:
\begin{align*}
T_i^{(j)}
(B,\dots, B)
&=\sum_{k_0+\dots+k_i=j-1}(A_0-z)^{-k_0-1}B(A_0-z)^{-k_1-1}\dots B(A_0-z)^{-k_i-1}\\
&\in \cL_{\frac{p+1}{j+i}}(\cH).
\end{align*}

(ii). Since $P_n\to 1$ in the strong operator topology, Lemma \ref{so_inLp} and inclusions \eqref{Ti_1}, \eqref{Ti_2} imply that 
\begin{align*}
\begin{split}
(A_0-z)^{-k_0-1}B_n(A_0-z)^{-k_1-1}&=P_n(A_0-z)^{-k_0-1}B(A_0-z)^{-k_1-1}P_n\\
 &\to (A_0-z)^{-k_0-1}B(A_0-z)^{-k_1-1} 
\end{split}
\end{align*}
 in $\cL_{\frac{p+1}{k_0+k_1+2}}(\cH)$
 and 
 \begin{equation*}\label{Ti_2_conv}
 B_n(A_0-z)^{-k_l-1}=P_nB(A_0-z)^{-k_l-1}P_n\\
 \to B(A_0-z)^{-k_l-1} 
 \end{equation*}
 in $\cL_{\frac{p+1}{k_l+1}}(\cH)$ for every $l=2,\dots i$. 
 Hence, using again the H\"older inequality, we conclude that 
{ $$\lim_{n\to\infty}\Big\|T_i^{(j)}(B_n,\dots, B_n)-T_i^{(j)}(B,\dots, B)\Big\|_{\frac{p+1}{j+i}}=0.$$}
 
 (iii). By Lemma \ref{lem_B_+_relative_dif_powers} we have that 
 $B(A_0-z)^{-1}\in\cL_{p+1}(\cH),$
 which implies that 
 $B(A_0-z)^{-k_l-1}\in \cL_{p+1}(\cH),$
 for any $k_l\in\bbZ_+$. 
 Therefore, by the H\"older inequality 
 $$B(A_0-z)^{-k_1-1}\dots B(A_0-z)^{-k_i-1}\in \cL_{\frac{p+1}{i}}(\cH),$$
that is 
 $R_i^{(j)}(B; B,\dots, B)\in \cL_{\frac{p+1}{i}}(\cH).$
 
 The proof of part (iv) can be obtained similarly to (ii) taking into account also that 
 $(A_{0}+B_n-z)^{-l}\to (A_0+B-z)^{-l},\quad l\in\bbN$
in the strong operator topology (see Lemma \ref{approx_A_n} (i)).
\end{proof}

The following theorem establishes, in particular, that the $p$-relative trace-class assumption implies that $A_1$ and $A_0$ are $p$-resolvent comparable  and the difference {$(A_0+B-z)^{-p}-(A_0-z)^{-p}$} can be approximated (in the trace-class norm) by $(A_0+P_nBP_n-z)^{-p}-(A_0-z)^{-p}$. The theorem is our first key result used in the proof of the principal trace formula. 

\begin{theorem}\label{thm_trace_class_approx}Assume that $B$ is a $p$-relative trace-class perturbation of $A_0$. Set $P_n=\chi_{[-n,n]}(A_0)$ and  $B_n:=P_nBP_n$ and let $z\in\bbC\setminus\bbR$. For any $j=1,\dots, p$ we have that 
$$\left(A_0+B-z\right)^{-j} - \left( A_0 -z \right)^{-j}\in \cL_{\frac{p+1}{j+1}}(\cH)$$ 
and $\left( A_0+B_n-z \right)^{-j} - \left( A_0 -z \right)^{-j}$ converges to $\left( A_0+B-z \right)^{-j} - \left( A_0 -z \right)^{-j}$ with respect to the norm  
of $\cL_\frac{p+1}{j+1}(\cH)$. 
For any $j>p$ the assertion holds in the trace-class ideal $\cL_1(\cH)$. 
\end{theorem}
\begin{proof}We prove the assertion for $j=1,\dots, p$ only, as the proof for $j>p$ is similar.

Let $j=1,\dots, p$ be fixed.
By \eqref{diff_resolvent_Ti} we can write 
\begin{align*}
\begin{split}
\left( A_0+B -z \right)^{-j}& - \left( A_0 -z \right)^{-j}
\\
&=T_1^{(j)}(B)+\dots +T_j^{(j)}(B,\dots, B)+R_{j+1}^{(j)}(B;B,\dots, B).
\end{split}
\end{align*}

Hence, the first assertion follows from Lemma  \ref{mappings_Ti_lem} (i) and (iii).
%

To prove  convergence, we use \eqref{diff_resolvent_Ti} with $B_n$ instead of $B$ to write 
\begin{align*}
(A_0&+B_n-z )^{-j} - ( A_0 -z)^{-j}\\
&=T_1^{(j)}(B_{n})+\dots +T_j^{(j)}(B_{n},\dots, B_{n})+R_{j+1}^{(j)}(B_{n};B_{n},\dots, B_{n}).\nonumber
\end{align*}
Appealing to Lemma \ref{mappings_Ti_lem} (ii) and (iv) we conclude the proof.
\end{proof}

Next, we proceed with the discussion of the right-hand side of the principal trace formula \eqref{ch_principla trace formula_intro_formula}. Let, as before, $B$ be a $p$-relative trace-class perturbation of a self-adjoint operator $A_0$. We introduce the straight-line path $\{A_s\}_{s\in[0,1]}$ joining $A_0$ and $A_0+B$, by setting
$A_s=A_0+sB,\quad s\in[0,1].$
We also introduce the path 
$\{A_{s,n}\}_{s\in[0,1]}$ joining $A_0$ and $A_0+B_n$ by
$$A_{s,n}=A_0+sB_n, \quad s\in[0,1], n\in\bbN.$$

We firstly note the following:

\begin{remark}\label{rem_norm_dif_res_independent}
Repeating the proof of Theorem \ref{thm_trace_class_approx} replacing $B$ and $B_n$ by the operators $sB$ and $sB_n$\, $s\in(0,1]$, respectively, one can conclude that the functions 
$s\mapsto \|\left(A_s-z\right)^{-j} - \left( A_0 -z \right)^{-j}\|_{\frac{p+1}{j+1}},$
and $s\mapsto \|\left( A_{s,n}-z \right)^{-j} - \left( A_0 -z \right)^{-j}\|_{\frac{p+1}{j+1}}$
are continuous with respect to $s$ and uniformly bounded with respect to $n\in\bbN$. $\diamond$  
\end{remark}

To show that the right-hand side of the principal trace formula \eqref{ch_principla trace formula_intro_formula} is well-defined we recall some results from the theory of double operator integration.  
We refer the reader to \cite{BS_DOI_I,BS_DOI_II,BS_DOI_III,PS_Crelle,PS_JOT}  for the precise definition of double operator integrals and their properties (see also the survey \cite{BS_DOI_2003}). 

Suppose that $A_0,A_1$ are self-adjoint (possibly unbounded) operators acting on $\cH$, and $f$ is a bounded Borel function on $\sigma(A_0)\times\sigma(A_1)$. Heuristically, the \emph{double operator integral} $T^{A_0,A_1}_f$ is a mapping acting on $\cB(\cH)$ and is defined using the spectral measures $E_{A_0}$ and $E_{A_1}$ of $A_0$ and $A_1$ by
\[T^{A_0,A_1}_f(B):=\int_{\sigma(A_1)}\int_{\sigma(A_0)}f(\lambda,\mu) \, dE_{A_0}(\lambda)\,B\,dE_{A_1}(\mu), \quad B\in\cB(\cH).\]
For a given function $f$ on $\sigma(A_0)\times\sigma(A_1)$, the double operator integral $T^{A_0,A_1}_f$ may or may not be a bounded operator on $\cB(\cH)$ (or on normed ideals of $\cB(\cH)$).
\begin{proposition}\cite{Ya05}\label{prop_final_DOI}
Let $A,B$ be self-adjoint operators. Assume that for some odd $m\in\bbN$  for all   $a\in\bbR\backslash\{0\}$, we have 
\begin{equation*}
\big[(B - a i)^{-m} - (A - a i)^{-m}\big] \in \cL_1\big(\cH\big)
\end{equation*} and let $f\in S(\bbR)$. Then there exist  double operator integrals $T_{f,a_1}^{A,B}$ and $T_{f,a_2}^{A,B}$, which are bounded on $\cL_1(\cH)$ and 
$$f(A)-f(B)=\sum_{j=1,2}T_{f,a_j}^{A,B}\big((A-a_ji)^{-m} - (B-a_ji)^{-m}\big)\in\cL_1(\cH).$$
\end{proposition}
\begin{theorem}\cite[Section 3]{CGLNPS16}\label{DOI_converge_pointwise}
Assume that $A,A_n, B,B_n$, $n\in\bbN$, are self-adjoint operators such that $A_n\to A$ and $B_n\to B$ as $n\to\infty$ in the strong resolvent sense.
In addition we assume that for some $m \in \bbN$, $m$ odd,  and every $a\in\bbR \backslash \{0\}$, 
$$\big[( A + ia)^{-m}
- ( B + ia)^{-m}\big],  \quad \big[( A_n + ia)^{-m} - ( B_n + ia)^{-m}\big] \in\cL_1(\cH), $$
and 
\begin{equation*}
\lim_{n \rightarrow \infty} \big\|\big[( A_n + ia)^{-m} - ( B_n + ia)^{-m}\big] - \big[( A + ia)^{-m}
- ( B + ia)^{-m}\big]\|_1 =0. 
\end{equation*} 

Then  for any $f\in S(\bbR)$ and the double operator integrals $T_{f,a_j}^{A,B}$ and $T_{f,a_j}^{A_n,B_n}$ as in Proposition \ref{prop_final_DOI} we have that 
$T_{f,a_j}^{A_n,B_n}\to T_{f,a_j}^{A,B},\quad j=1,2,$
pointwise on $\cL_1(\cH)$. In particular, 
\begin{equation*}
\lim_{n \rightarrow \infty} \big\| [f(A_n) - f(B_n)] - [f(A)
- f(B)]\big\|_{\cL_p(\cH)}=0.  
\end{equation*} 

\end{theorem}
\begin{proposition}
\label{principla trace formula_right-hand side_well_defined}
Let $B$ be a $p$-relative trace-class perturbation of a self-adjoint operator $A_0$ and let $A_s=A_0+sB, s\in[0,1]$. The function 
$s\mapsto e^{-tA_s^2}B, \quad t>0,$ is a continuous $\cL_1(\cH)$-valued function on $[0,1]$. In particular, the integral
$$\int_0^1\tr\Big(e^{-tA_s^2}B\Big)ds$$ is well-defined. 
\end{proposition}

\begin{proof}
Firstly we  show that the operator $e^{-tA_s^2}B$ is a trace class operator for any fixed $s\in [0,1]$. It is sufficient to show that  the operator
$(A_s+i)^{-p-1}B$ is a trace-class operator.
We may write 
\begin{align*}
(A_s+i)^{-p-1}B=\Big((A_s+i)^{-p-1}-(A_0+i)^{-p-1}\Big)B+(A_0+i)^{-p-1}B.
\end{align*}
As $B$ is a $p$-relative trace-class perturbation of $A_0$ it follows that  the second term on the right-hand side is a trace-class operator. For the first term, Theorem \ref{thm_trace_class_approx} implies the operator $(A_s+i)^{-p-1}-(A_0+i)^{-p-1}$ is a trace-class operator. Hence, $(A_s+i)^{-p-1}B\in\cL_1(\cH)$ for any $s\in[0,1]$, as required.

Now, we prove that the mapping $s\mapsto e^{-tA_s^2}B$ is continuous in $\cL_1(\cH)$-norm. Let $s_1,s_2\in [0,1]$. We set 
$p_0=2\lfloor \frac p2\rfloor+1.$
By Theorem \ref{thm_trace_class_approx}  the operators $A=A_{s_1}$ and $B=A_{s_2}$ satisfy the assumption of Proposition~\ref{prop_final_DOI}. Therefore, we have 
\begin{align}\label{diff_heat_via_resolvents}
\begin{split}
e^{-tA_{s_1}^2}B-e^{-tA_{s_2}^2}B=\sum_{j=1,2}T_{f,{a_j}}^{A_{s_1},A_{s_2}}\Big({(A_{s_1}-a_j i)^{-p_0}-(A_{s_2}-a_j i)^{-p_0}}\Big)\cdot B,
\end{split}
\end{align}
where $f(x)=e^{-tx^2},x\in\bbR, t>0$. 

By Remark \ref{rem_norm_dif_res_independent} we have that 
$$\Big\|(A_{s_1}-a_j i)^{-p_0}-(A_{s_2}-a_j i)^{-p_0}\Big\|_1\to 0,\quad \text{as } s_1-s_2\to 0.$$
Furthermore, by Theorem \ref{DOI_converge_pointwise} the double operator integral $T_{f,{a_j}}^{A_{s_1},A_{s_2}},$ $j=1,2$,  converges pointwise on $\cL_1(\cH)$ to $T_{f,{a_j}}^{A_{s_1},A_{s_1}},$ as $s_2\to s_1$. Therefore, 
\begin{align*}
\Big\|T_{f,{a_j}}^{A_{s_1},A_{s_2}}&\Big((A_{s_1}-a_j i)^{-p_0}-(A_{s_2}-a_j i)^{-p_0}\Big)\Big\|\\
&\leq \Big\|\big(T_{f,{a_j}}^{A_{s_1},A_{s_2}}-T_{f,{a_j}}^{A_{s_1},A_{s_1}}\big)\Big((A_{s_1}-a_j i)^{-p_0}-(A_{s_2}-a_j i)^{-p_0}\Big)\Big\|\\
&\qquad+\Big\|T_{f,{a_j}}^{A_{s_1},A_{s_1}}\Big((A_{s_1}-a_j i)^{-p_0}-(A_{s_2}-a_j i)^{-p_0}\Big)\Big\|\\
&\to 0, \quad s_1-s_2\to 0, \quad j=1,2.
\end{align*}
Thus, equality \eqref{diff_heat_via_resolvents} implies the required result:
$$\|e^{-tA_{s_1}^2}B-e^{-tA_{s_2}^2}B\|_1\to 0,\quad s_1-s_2\to 0.$$
\end{proof}

To conclude this section we prove that the integral in Proposition \ref{principla trace formula_right-hand side_well_defined} can be approximated by a similar integral with $B$ replaced by $B_n$ (and so $A_s$ replaced by $A_{s,n}$).

\begin{proposition}\label{prop_conv_integral_exp}
We have that 
\begin{equation}\label{conv_right-hand side_principla trace formula}
\lim_{n\to\infty}\int_0^1\tr\Big(e^{-tA_{s,n}^2}B_n\Big)ds=\int_0^1\tr\Big(e^{-tA_{s}^2}B\Big)ds.
\end{equation}
\end{proposition} 
\begin{proof}
Let $s\in[0,1]$ be fixed. { We write
$$e^{-tA_{s,n}^2}B_n=(A_{s,n}+i)^{p+1}e^{-tA_{s,n}^2} \cdot (A_{s,n}+i)^{-p-1}B_{n}.$$
Since $A_{s,n}\to A_s$ in the strong resolvent sense (see Lemma \ref{approx_A_n}) and the function $x\mapsto e^{-tx^2}(x+i)^{p+1}, x\in\bbR,$ is continuous and bounded, \cite[Theorem VIII.23]{RS_book_I} implies that $(A_{s,n}+i)^{p+1}e^{-tA_{s,n}^2}\to (A_{s}+i)^{p+1}e^{-tA_{s}^2}$} in the strong operator topology. Hence, by Lemma~\ref{so_inLp}, to prove the convergence 
\begin{equation}\label{conv_heat_wrt_n}
\lim_{n\to\infty} \Big\|e^{-tA_{s,n}^2}B_n-e^{-tA_{s}^2}B\Big\|_1=0
\end{equation}
 it is sufficient to show that 
\begin{equation}\label{conv_heat_via_resolvent}
\lim_{n\to\infty}\Big\|(A_{s,n}+i)^{-p-1}B_{n}-(A_{s}+i)^{-p-1}B\Big\|_1=0.
\end{equation}

 To prove  \eqref{conv_heat_via_resolvent} we  write 
$$ (A_{s,n}
+i)^{-p-1}B_{n}
 =\Big((A_{s,n}+i)^{-p-1}-(A_{0}+i)^{-p-1}\Big)B_{n}+P_n(A_{0}+i)^{-p-1}BP_n
 $$
From Theorem \ref{thm_trace_class_approx} the difference
$\Big((A_{s,n}+i)^{-p-1}-(A_0+i)^{-p-1}\Big)$ converges to $\Big((A_{s}+i)^{-p-1}-(A_0+i)^{-p-1}\Big)$ in $\cL_{1}(\cH)$. Moreover,
combining the assumption that ${(A_0+i)^{-p-1}}B\in\cL_1(\cH)$ with the strong operator convergence $P_n\to 1$  and with Lemma \ref{so_inLp},  we obtain  
$P_n{(A_0+i)^{-p-1}}BP_n\to {(A_0+i)^{-p-1}}B$
in $\cL_1(\cH)$. Hence, 
{ for every fixed $s\in[0,1]$, we conclude that the sequence $\{(A_{s,n}+i)^{-p-1}B_{n}\}_{n\in\bbN}$ converges to $(A_{s}+i)^{-p-1}B$ in $\cL_1(\cH)$, which suffices to prove \eqref{conv_heat_wrt_n}.}

We claim that the sequence of functions
$s\mapsto \|e^{-tA_{s,n}^2}B_n\|_1,\quad s\in[0,1],\quad t>0,$
is uniformly bounded (with respect to $n$) by a continuous function.
Indeed, we have 
\begin{align*}
\|e^{-tA_{s,n}^2}B_n\|_1&\leq \|(A_{s,n}+i)^{p+1}e^{-tA_{s,n}^2}\|\cdot \|(A_{s,n}+i)^{-p-1}B_n\|_1\\
&\leq \const \|\big((A_{s,n}+i)^{-p-1}-(A_0+i)^{-p-1}\big)B_n\|_1\\
&\qquad +\const \|(A_0+i)^{-p-1}B_n\|_1,
\end{align*}
where the constant is independent of $s$ and $n$. 

By Remark \ref{rem_norm_dif_res_independent} the first term in the previous inequality involves a sequence of functions uniformly majorised by a  continuous function. As $B_n=P_nBP_n$, the second term is uniformly majorised by a constant: 
$\const \|(A_0+i)^{-p-1}B\|_1.$

Thus, appealing to  \eqref{conv_heat_wrt_n} and the dominated convergence theorem we infer that 
\begin{equation*}
\lim_{n\to\infty}\int_0^1\tr\Big(e^{-tA_{s,n}^2}B_n\Big)ds=\int_0^1\tr\Big(e^{-tA_{s}^2}B\Big)ds.
\end{equation*}
\end{proof}

\section{The setting for our main result}\label{ch_approximation}

In the present section we introduce the set-up for the rest of the paper as well as explain the  approximation scheme we employ in our approach. The idea is to introduce a spectral `cut-off' by the characteristic function of the interval $[-n,n]$ for positive integral $n$.
We use the subscript $n$ on the operators introduced in the current section  to indicate these spectrally cut-off or `reduced' versions.

 Our paths are restricted by the final Hypothesis  \ref{hyp_final}. In this introductory discussion however we will work under the less restrictive Hypothesis \ref{hyp_initial}.

\begin{hypothesis} \label{hyp_initial} 
$(i)$ Assume $A_-$ is self-adjoint on $\dom(A_-) \subseteq \cH$. \\
$(ii)$ Suppose we have a family of bounded self-adjoint operators 
$\{B(t)\}_{t \in \bbR} \subset \cB(\cH)$, continuously differentiable with respect to $t$ in the uniform operator norm,  
such that 
\begin{equation}
\|B'(\cdot)\| \in L^1(\bbR)\cap L^\infty(\bbR).     \label{intB'}
\end{equation}
$(iii)$ Suppose that for some $p\in\bbN\cup\{0\}$ we have 
\begin{equation}\label{def_p_relative}
B'(t)(A_-+i)^{-p-1}\in\cL_1(\cH),\quad \int_\bbR \|B'(t)(A_-+i)^{-p-1}\|_1dt<\infty.
\end{equation}
In what follows, we always choose the smallest $p\in\bbN\cup\{0\}$ which satisfies \eqref{def_p_relative}. 
\end{hypothesis}

Given Hypothesis \ref{hyp_initial} we introduce the family of self-adjoint operators 
$A(t)$, $t \in \bbR$, in $\cH$, by 
\begin{equation*}
A(t) = A_- + B(t), \quad \dom(A(t)) = \dom(A_-), \quad t \in \bbR.
\end{equation*}
Writing
\begin{equation}
\label{Bt_as_integral}
B(t) = B(t_0) + \int_{t_0}^t  B'(s)\, ds, \quad t, t_0 \in \bbR, 
\end{equation}
with the convergent Bochner integral on the right-hand side, we conclude that the self-adjoint asymptotes (with respect to the operator norm)
\begin{equation}\label{B_pm}
\lim_{t \to \pm \infty} B(t) := B_{\pm} \in \cB(\cH)     
\end{equation}  
exist. Purely for convenience of notation, we will make the choice
$B_- = 0$
in the following and also introduce the asymptote, 
\begin{equation*}
A_+ = A_- + B_+, \quad \dom(A_+) = \dom(A_-).    
\end{equation*}
Assumption \eqref{intB'}  and equality \eqref{Bt_as_integral} also yield,
\begin{equation}
\sup_{t \in \bbR} \|B(t)\| \leq \int_{\bbR}  \|B'(t)\| dt< \infty.     \label{supB} 
\end{equation}

A simple application of the resolvent identity yields (with $t \in \bbR$, 
$z \in \bbC \backslash \bbR$)
\begin{align*}
& (A(t) - z I)^{-1} = (A_{\pm} - z I)^{-1} 
- (A(t) - z I)^{-1} [B(t) - B_{\pm}] (A_{\pm} - z I)^{-1},   \\
& \big\|(A(t) - z I)^{-1} - (A_{\pm} - z I)^{-1}\big\| \leq 
|\Im(z)|^{-2} \|B(t) - B_{\pm}\|,
\end{align*}
and hence proves that 
$\lim_{t \to \pm \infty} (A(t) - z I)^{-1} = (A_{\pm} - z I)^{-1}, \quad 
z \in \bbC \backslash \bbR$
with respect to the operator norm $\|\cdot\|$. 

Repeating  the argument of \cite[(3.49)]{GLMST}
one can prove that 
\begin{equation}\label{B_+_p-relative}
B_+(A_-+i)^{-1-p},\quad B(t)(A_-+i)^{-1-p}\in\mathcal{L}_1(\cH),
\end{equation}
that is, $B_+$ as well as the family $\{B(t)\}_{t\in\bbR}$ are $p$-relative trace-class perturbations with respect to $A_-$. In particular, results of Section \ref{ch_prelim} apply to the perturbations $B_+$ and $B(t), t\in\bbR$, of the operator $A_-$. 

As the next step in this section, we introduce the key technical ideas that enable us to use the old results of \cite{GLMST} in an approximation
scheme.
As in Section \ref{ch_prelim} we introduce  a spectral `cut-off' of the operator $A_-$ by setting
$P_n=\chi_{[-n,n]}(A_-).$

Let $\{B(t)\}_{t\in\bbR}$ be a one parameter family of perturbations of $A_-$ satisfying Hypothesis~\ref{hyp_initial_+}.
We introduce the family $\{B_n(t)\}_{t\in\bbR}$, $n\in\bbN$,  of reduced operators by setting
\begin{equation*}
B_n(t):=P_nB(t)P_n,\quad t\in\bbR, n\in\bbN.
\end{equation*}
In this case, 
\begin{align} \label{Ant}
\begin{split} 
& A_n(t):= A_- + B_n(t), \quad \dom(A_n(t)) = \dom(A_-), \quad n \in \bbN, \; t \in \bbR. 
\end{split} 
\end{align}

In particular, one concludes that 
\begin{equation}\label{B+n}
B_{+,n}:=\lim_{t\to+\infty} B_n(t)=P_nB_+P_n, 
\end{equation}
in the uniform operator norm. Therefore for the reduced asymptotes $A_{+,n}$, constructed with the family $\{B(t)\}_{t\in\bbR}$ replaced by $\{B_n(t)\}_{t\in\bbR}$, we obtain 
\begin{align*}
A_{+,n}:= A_- + B_{+,n}=A_-+P_nB_+P_n, \quad \dom(A_{+,n}) = \dom(A_-).
\end{align*}
We note that the equality \eqref{B_+_p-relative} together with the definition of the projections $P_n$ implies that
$B_{+,n}=P_nB_+P_n\in \cL_1(\cH).$

The following proposition shows that the family $\{B_n(t)\}_{t\in\bbR}$ of `approximants' consists of trace-class operators, and so for this family the results of \cite{Pu08, GLMST, CGPST_Witten} hold. The proof of this proposition is a verbatim repetition of the proof of \cite[Proposition 2.3]{GLMST} and is therefore omitted. 
\begin{proposition}\label{red_sat_Push}
The family $\{B_n(t)\}_{t\in\bbR}$ consists of trace-class perturbations of $A_-$ and therefore satisfies the assumption in \cite{Pu08} and \cite{GLMST}.
\end{proposition}




By $L^2(\bbR,\cH)$ we denote the Hilbert space of all $\cH$-valued Bochner square integrable function on $\bbR$.
Linear operators in the Hilbert space  $L^2(\bbR; \cH)$, will be denoted by calligraphic boldface symbols of the type $\bsT$, to distinguish them from 
operators $T$ in $\cH$. In particular, operators denoted by 
$\bsT$ in the Hilbert space $L^2(\bbR;\cH)$ typically represent operators associated with a 
family of operators $\{T(t)\}_{t\in\bbR}$ in $\cH$, defined by
\begin{align}
&(\bsT f)(t) = T(t) f(t) \, \text{ for a.e.\ $t\in\bbR$,}    \no \\
& f \in \dom(\bsT) = \big\{g \in L^2(\bbR;\cH) \,\big|\,
g(t)\in \dom(T(t)) \text{ for a.e.\ } t\in\bbR;    \lb{1.1}  \\
& \quad t \mapsto T(t)g(t) \text{ is (weakly) measurable;} \, 
\int_{\bbR} \|T(t) g(t)\|_{\cH}^2 \, dt <  \infty\bigg\}.   \no
\end{align}

Let $\bsA_-$ be the  operator acting in $L^2(\bbR;\cH)$ defined by  \eqref{1.1} with a constant fibre family $\{A_-(t)\}_{t\in\bbR}=\{A_-\}_{t\in\bbR}$. Similarly, let the operators $\bsA$, $\bsB, \bsA '=\bsB', $ be defined by \eqref{1.1} in terms of the families $\{A(t)\}_{t\in\bbR}$, $\{B(t)\}_{t\in\bbR}$, and $\{B'(t)\}_{t\in\bbR}$, respectively. 
Since $B(t), B'(t)$ are bounded operators for every $t\in\bbR$ and $\|B(\cdot)\|, \|B'(\cdot)\|\in L^\infty(\bbR)$ (see \eqref{supB} and Hypothesis \ref{hyp_initial}  (ii), respectively) we have that 
$\bsB, \, \bsB'\in\cB(L^2(\bbR;\cH)).$
Since, in addition, $A(t)=A_-+B(t),$
we infer that 
$$\bsA=\bsA_-+\bsB,\quad \dom(\bsA)=\dom(\bsA_-).$$


Now, to introduce the operator $\bsD_\bsA^{}$ in $L^2(\bbR;\cH)$, we recall that 
the operator $d/dt$ in $L^2(\bbR;\cH)$  is defined by 
\begin{align*}
& \bigg(\f{d}{dt}f\bigg)(t) = f'(t) \, \text{ for a.e.\ $t\in\bbR$,}    \no \\
& \, f \in \dom(d/dt) = \big\{g \in L^2(\bbR;\cH) \, \big|\,
g \in AC_{\loc}\big(\bbR; \cH\big), \, g' \in L^2(\bbR;\cH)\big\}   \no \\
& \hspace*{2.3cm} = W^{1,2} \big(\bbR; \cH\big).     
\end{align*} 
Then, the operator $\bsD_\bsA^{}$ is defined by setting 
\begin{equation}\label{D_A}
\bsD_\bsA^{} = \f{d}{dt} + \bsA,
\quad \dom(\bsD_\bsA^{})= W^{1,2}(\bbR; \cH) \cap \dom(\bsA_-).   
\end{equation}

Assuming Hypothesis \ref{hyp_initial} and repeating  the proof of \cite[Lemma 4.4]{GLMST} one can show that the operator
$\bsD_\bsA^{}$ is densely defined and closed in $L^2(\bbR; \cH)$. Furthermore,
the adjoint operator $\bsD_\bsA^*$ of $\bsD_\bsA^{}$ in $L^2(\bbR; \cH)$ is then given
by (cf.\ \cite{GLMST})
\begin{equation*}
\bsD_\bsA^*=- \f{d}{dt} + \bsA, \quad
\dom(\bsD_\bsA^*) = W^{1,2}(\bbR; \cH) \cap \dom(\bsA_-).   
\end{equation*}

This enables us to introduce the non-negative, self-adjoint operators 
$\bsH_j$, $j=1,2$, in $L^2(\bbR;\cH)$ by
\begin{equation}\label{def_Hj}
\bsH_1 = \bsD_{\bsA}^{*} \bsD_{\bsA}^{}, \quad 
\bsH_2 = \bsD_{\bsA}^{} \bsD_{\bsA}^{*}.
\end{equation}

The following result is proved in \cite[Theorem~2.6]{CGPST_Witten} under a relatively trace class perturbation assumption.  It was already noted in  \cite[Remark 2.7]{CGPST_Witten} that the result holds without this assumption. Thus, in our more general setting the following theorem holds. 

\begin{theorem}\label{t8.iff}
Assume Hypothesis \ref{hyp_initial}. Then the operator $\bsD_\bsA^{}$ is Fredholm if and only if  
$0 \in \rho(A_+) \cap \rho(A_-)$.
\end{theorem}

Next, we turn to the reduced counterparts $\bsH_{j,n}, j=1,2, n\in\bbN$, of the operators $\bsH_j, j=1,2$. 
Recall that the family $\{B_n(t)\}_{t\in\bbR}, n\in\bbN$ is defined by (see \eqref{B+n})
$$B_n(t)=P_nB(t)P_n, \quad P_n=\chi_{[-n,n]}(A_-).$$
In this case, the corresponding operator $\bsA_n$ is defined as 
$\bsA_n=\bsA_-+\bsB_n,$
where $\bsB_n$ is defined by \eqref{1.1} with $\{T(t)\}_{t\in\bbR}=\{B_n(t)\}_{t\in\bbR}$. 

Denote by $\bsH_{j,n}, \, j=1,2,$ the operator defined by \eqref{def_Hj} with $\bsD_\bsA$ replaced by the corresponding operator $\bsD_{\bsA_n}=\frac{d}{dt}+\bsA_n.$

\section{Some  uniform norm estimates }\label{sec_unifbound}

In this section we prove some uniform norm estimates for the operators $\bsH_j$, $j=1,2$. This result will be used in the proof of the second key result of this Section, Theorem~\ref{left-hand side_conver_in_L1}. The methods are borrowed from the abstract pseudo-differential calculus of non-commutative geometry \cite{CGRS_memoirs, CGPRS15, Connes1995}.

For future purposes we also introduce $\bsH_0$ in $L^2(\bbR; \cH)$ by  
\begin{equation}\label{def_H0}
\bsH_0 = - \f{d^2}{dt^2} + \bsA_-^2, \quad \dom(\bsH_0) 
= W^{2,2}(\bbR; \cH) \cap \dom\big(\bsA_-^2\big).     
\end{equation}
By \cite[Theorem VIII.33]{RS_book_I}, the operator $\bsH_0$ is self-adjoint and positive. We note that the operators $\bsA_-$ and $\bsH_0$ commute and 
\begin{equation}\label{domH_0^12}
\dom\bsH_0^{1/2}=\dom(d/dt)\cap \dom \bsA_-.
\end{equation}

The proof of the following result can be found in \cite[Lemma 4.7]{GLMST}. Observe, that the proof given there does not require the full strength of the assumptions made in that paper.
The statement is formulated using  Hypothesis \ref{hyp_initial}. In fact, it requires only Hypothesis \ref{hyp_initial} (i).

\begin{lemma}\label{lem_A_H_0_bounded}
\cite[Lemma 4.7]{GLMST} 
For every $z<0$, the operator $\bsA_-(\bsH_0-z)^{-1/2}$ is bounded and 
$\|{\bsA_-}{(\bsH_0-z)^{-1/2}}\|\leq 1,\quad z<0.$
\end{lemma}

In the context of Lemma \ref{lem_A_H_0_bounded}, we note also that 
\begin{equation}\label{d/dt_H_0_bounded}
\frac{d}{dt}\cdot (\bsH_0+1)^{-1/2}\in\cB(L^2(\bbR;\cH)).
\end{equation}

In what follows, we need to strengthen Hypothesis \ref{hyp_initial} as follows.
\begin{hypothesis} \label{hyp_initial_+} 
In addition to Hypothesis \ref{hyp_initial}, 
assume that $\dom(\bsH_0^{1/2})$ is invariant with respect to the operator $\bsB$.
\end{hypothesis}

Assuming Hypothesis \ref{hyp_initial_+} in the following, we have that 
$\bsA_-\bsB$ is an operator well defined on $\dom\bsH_0^{1/2}$, since 
$\dom\bsH_0^{1/2}\subset \dom\bsA_-$ (see \eqref{domH_0^12}).
Therefore, recalling that $\bsA=\bsA_-+\bsB$ one can decompose 
$\bsH_j$, $j=1,2$, as follows
\begin{align}\label{H_j_decomposition}
\begin{split}
 \bsH_j& = -\f{d^2}{dt^2} + \bsA^2 + (-1)^j \bsA^{\prime}    \\
& = \bsH_0 + \bsB \bsA_- + \bsA_- \bsB + \bsB^2 + (-1)^j \bsB^{\prime},  \\
& \dom(\bsH_j) = \dom(\bsH_0), \quad j =1,2. 
\end{split}
\end{align}

Using the standard resolvent identity, we obtain
$$(\bsH_j-z\bsI)^{-1}-(\bsH_0-z\bsI)^{-1}=-(\bsH_j-z\bsI)^{-1}(\bsH_j-{\bsH}_0)(\bsH_0-z\bsI)^{-1}$$
\begin{align}\label{ssss}=-(\bsH_j-z\bsI)^{-1}(\bsB\bsA_-+\bsA_-\bsB+\bsB^2+(-1)^j\bsB')(\bsH_0-z\bsI)^{-1},
\end{align}
for $j=1,2$ and $z\in\bbC\setminus\bbR_+$.

Assuming Hypothesis \ref{hyp_initial_+}, one obtains, similarly to the proof of \eqref{H_j_decomposition}, the decompositions for the operators $\bsH_{j,n}, \, j=1,2,$ of the following form
\begin{align*}
\begin{split}
& \bsH_{j,n} = \f{d^2}{dt^2} + \bsA_n^2 + (-1)^j \bsA_n^{\prime} \\
& \hspace*{8mm} = \bsH_0 + \bsB_n \bsA_- + \bsA_- \bsB_n 
+ \bsB_n^2 + (-1)^j \bsB_n^{\prime}, \\
& \dom(\bsH_{j,n}) = \dom(\bsH_0) = W^{2,2}(\bbR) \cap \dom\big(\bsA_-^2\big),  
\quad n \in \bbN, \;  j =1,2, 
\end{split} 
\end{align*}
with
$\bsB_n = \bsP_n \bsB \bsP_n, \quad 
\bsB_n^{\prime} = \bsP_n \bsB^{\prime} \bsP_n, \quad n \in \bbN,$
 $\bsP_n=\chi_{[-n,n]}(\bsA_-)=1\otimes P_n.$

The following result can be found in \cite[Lemma 3.12 (i)]{CGLS16}.
\begin{lemma}\label{approx_H_jn}
Assume Hypothesis \ref{hyp_initial_+}. The operators $\bsH_{j,n}$ converge to $\bsH_j$, $j=1,2$, in the strong resolvent sense.
\end{lemma}

For $k\in\bbN$ we introduce
\begin{gather*}{\rm dom}(\delta_{\bsH_0}^k)=\{\bsT\in\cB(L^2(\bbR;\cH)):\bsT\dom(\bsH_0^{j/2})\subset \dom(\bsH_0^{j/2}), \forall j=1,\dots,k, \\
\quad \text{ and the operator } [(1+\bsH_0)^{1/2},\bsT]^{(k)}, \text{ defined on } {\rm dom}(\bsH_0^{k/2})\\
\quad \text{ extends to a bounded operator on } L^2(\bbR;\cH)\}.
\end{gather*}
and set 
\begin{equation}\label{def_deriv_bsH_0}
\delta_{\bsH_0}^k(\bsT)=\overline{[(1+\bsH_0)^{1/2},\bsT]^{(k)}},\quad \bsT\in\dom\delta_{\bsH_0}^k.
\end{equation}
where the notation $[(1+\bsH_0)^{1/2},\bsT]^{(k)}$ stands for $k$-th repeated commutator defined by
\begin{align*}
[(1+\bsH_0)^{1/2},T]^{(k)}&=[(1+\bsH_0)^{1/2},\dots[(1+\bsH_0)^{1/2}, [(1+\bsH_0)^{1/2},T]]\dots],\\
&\quad \dom([(1+\bsH_0)^{1/2},T]^{(k)})=\dom(\bsH_0^{k/2}).
\end{align*}
For convenience, we set 
$[(1+\bsH_0)^{1/2},\bsT]^{(0)}=\bsT.$

\begin{remark}\label{rem_dom_delta_0_subalgera}
Note that $\bigcap_{j=0}^k\dom(\delta_{\bsH_0}^j), k\in\bbN,$ is a  subalgebra in $\cB(L^2(\bbR;\cH))$.  $\diamond$
\end{remark}

We note that if
 $\bsT\in\dom(\delta_{\bsH_0})$, then for every $\xi\in\dom(\bsH_0^{1/2})$ we have
 \begin{align*}
\bsT(\bsH_0+1)^{1/2}\xi=(\bsH_0+1)^{1/2}\bsT\xi-[(\bsH_0+1)^{1/2},\bsT]\xi.
 \end{align*}
  Hence, if $\bsT\in\bigcap_{j=1}^k\dom(\delta_{\bsH_0}^j)$, for some $k\in\bbN$, then for every $\xi\in\dom(\bsH_0^{k/2})$, using this equality repeatedly, we obtain 
\begin{align}\label{big_BH_0}
\begin{split}
\bsT&(\bsH_0+1)^{k/2}\xi=\bsT(\bsH_0+1)^{1/2}(\bsH_0+1)^{\frac{k-1}2}\xi\\
&=(\bsH_0+1)^{1/2}\bsT(\bsH_0+1)^{\frac{k-1}2}\xi-[(\bsH_0+1)^{1/2},\bsT](\bsH_0+1)^{\frac{k-1}2}\xi\\
&=\dots\\
&=\sum_{j=0}^k(-1)^{j}C_k^j(\bsH_0+1)^{j/2}[(\bsH_0+1)^{1/2},\bsT]^{(k-j)}\xi,
\end{split}
\end{align}
where $C_k^j$ denotes the binomial coefficient. 

\begin{lemma}\label{big_estimate_on_norm_k=1}
Assume $\bsB\in\dom(\delta_{\bsH_0})$. Then the operator ${(\bsH_i+1)^{-1/2}}(\bsH_0+1)^{1/2}, i=1,2,$ defined on $\dom(\bsH_0^{1/2})$ extends to a bounded operator on $L^2(\bbR;\cH)$ and 
$$\Big\|\overline{{(\bsH_i+1)^{-1/2}}(\bsH_0+1)^{1/2}}\Big\|\leq\const\cdot (\|\delta_{\bsH_0}(\bsB)\|+\|\bsB\|+\|\bsB\|^2+\|\bsB'\|).$$
\end{lemma}

\begin{remark}Note that the first assertion in Lemma \ref{big_estimate_on_norm_k=1} follows immediately from the closed graph theorem and the fact that $\dom(\bsH_0)=\dom(\bsH_i), i=1,2$. However, we also need an estimate on the uniform norm of this operator. $\diamond$
\end{remark}
\begin{proof}
Since the operator $\bsH_1$ is self-adjoint and positive we can write
$${(\bsH_1+1)^{-1/2}}=\frac1\pi\int_0^\infty\frac{d\lambda}{\lambda^{1/2}}{(1+\lambda+\bsH_1)^{-1}},$$
with the  RHS being a convergent Bochner integral (see e.g. \cite[p. 282]{Kato}).

By the resolvent identity \eqref{ssss} 
 for all $\xi\in\dom(\bsH_0)^{1/2}$ we have
\begin{align}\label{big_k=1}
&{(\bsH_1+1)^{-1/2}}(\bsH_0+1)^{1/2}\xi\nonumber\\
&\quad=\xi+\frac1\pi\int_0^\infty\frac{d\lambda}{\lambda^{1/2}}
({1+\lambda+\bsH_1})^{-1}\bsA_-\bsB{(\bsH_0+1)^{1/2}}{(1+\lambda+\bsH_0)^{-1}}\xi\nonumber\\
&\quad\quad+\frac1\pi\int_0^\infty\frac{d\lambda}{\lambda^{1/2}}
({1+\lambda+\bsH_1})^{-1}\bsB{\bsA_-^{}}{(\bsH_0+1)^{1/2}}{(1+\lambda+\bsH_0)^{-1}}\xi\nonumber\\
&\quad\quad+\frac1\pi\int_0^\infty\frac{d\lambda}{\lambda^{1/2}}
({1+\lambda+\bsH_1})^{-1}(\bsB^2-\bsB'){(\bsH_0+1)^{1/2}}{(1+\lambda+\bsH_0)^{-1}}\xi\nonumber\\
&\quad = \xi+I_1\xi+I_2\xi+I_3\xi.
\end{align}

The estimates
$$\Big\|(\bsH_0+1)^{1/2}({1+\lambda+\bsH_0})^{-1}\Big\|\leq(1+\lambda)^{-1/2},\quad \Big\|({1+\lambda+\bsH_1})^{-1}\Big\|\leq(1+\lambda)^{-1},$$
together with the fact that the operators $\bsB,\bsB'$ are bounded imply that the integral $I_3$ on the right-hand side of \eqref{big_k=1} converges in the uniform norm and 
$$\|I_3\|\leq \const(\|\bsB\|^2+\|\bsB'\|).$$

Similarly for $I_2$, using in addition Lemma \ref{lem_A_H_0_bounded}, we have 
$$\Big\|{\bsA_-}{(\bsH_0+1)^{1/2}}{(1+\lambda+\bsH_0)^{-1}}\Big\|\leq(1+\lambda)^{-1/2},$$
guaranteeing that  the integral $I_2$ converges in the operator norm and 
$\|I_2\|\leq \const \cdot \|\bsB\|.$

Finally, for the integral $I_1$ we write 
\begin{align*}
I_1&=\frac1\pi\int_0^\infty\frac{d\lambda}{\lambda^{1/2}}
({1+\lambda+\bsH_1})^{-1}{\bsA_-}\bsB\, {(\bsH_0+1)^{1/2}}{(1+\lambda+\bsH_0)^{-1}}\xi\\
&=\frac1\pi\int_0^\infty\frac{d\lambda}{\lambda^{1/2}}
({1+\lambda+\bsH_1})^{-1}{\bsA_-}{(\bsH_0+1)^{-1/2}} \bsB{(\bsH_0+1)}{(1+\lambda+\bsH_0)^{-1}}\xi\\
&\quad+\frac1\pi\int_0^\infty\frac{d\lambda}{\lambda^{1/2}}
({1+\lambda+\bsH_1})^{-1}{\bsA_-}{(\bsH_0+1)^{-1/2}} \\
&\quad \quad\quad\quad\quad\quad\cdot[(\bsH_0+1)^{1/2},\bsB]\,{(\bsH_0+1)^{1/2}}{(1+\lambda+\bsH_0)^{-1}}\xi.
\end{align*}
Since, $\bsB\in\dom(\delta_{\bsH_0})$, the operator $[(\bsH_0+1)^{1/2},\bsB]$ extends to a bounded operator on $L^2(\bbR;\cH)$. Hence, repeating the argument above, we conclude that $I_1$ is a bounded operator with
$\|I_1\|\leq \const \cdot (\|\bsB\|+\|\delta_{\bsH_0}(\bsB)\|).$

Thus, by \eqref{big_k=1} we have that the operator ${(\bsH_0+1)^{-1/2}}(\bsH_1+1)^{1/2}$ extends to a bounded operator on $L^2(\bbR;\cH)$ and 
$$\Big\|\overline{{(\bsH_0+1)^{-1/2}}(\bsH_1+1)^{1/2}}\Big\|\leq\const(1+\|\delta_{\bsH_0}(\bsB)\|+\|\bsB\|+\|\bsB\|^2+\|\bsB'\|).$$
\end{proof}

The following result will be used later in the proof of the convergence of the left-hand side of the principal trace formula. 
\begin{proposition}\label{big_estimate_on_norm}Assume $\bsB,\bsB'\in\bigcap_{j=1}^{k-1}\dom(\delta_{\bsH_0}^j)$ for some $k\geq 2$. Then the operator ${(\bsH_i+1)^{-k/2}}(\bsH_0+1)^{k/2}, i=1,2,$ defined on $\dom(\bsH_0^{k/2})$ extends to a bounded operator on $L^2(\bbR;\cH)$ and 
$$\Big\|\overline{{(\bsH_i+1)^{-k/2}}(\bsH_0+1)^{k/2}}\Big\|\leq\const\cdot Q(\|\delta_{\bsH_0}^j(\bsB)\|,\|\delta_{\bsH_0}^j(\bsB')\|), $$
for all $j=0,\dots, k-1 $ and for some polynomial $Q$ with positive coefficients.
\end{proposition}

\begin{proof}
We prove the assertion only for $i=1$, since the proof for $i=2$ is identical. 

We proceed by induction on $k$. For $k=1$ the assertion is proved in Lemma \ref{big_estimate_on_norm_k=1}. 
Let $k=2$. By the resolvent identity \eqref{ssss} we have 
\begin{align}\label{big_estimate_on_norm_k=2}
\begin{split}
({\bsH_1+1})^{-1}&(\bsH_0+1)\xi\\
&=\xi-({\bsH_1+1})^{-1}\bsB\bsA_-\xi-({\bsH_1+1})^{-1}\bsA_-\bsB\xi
\\
&\qquad-({\bsH_1+1})^{-1}(\bsB^2-\bsB')\xi
\end{split}
\end{align}
for all $\xi\in\dom(\bsH_0)$.
For the second term, using the fact that $\bsA$ and $\bsH_0$ commute, we write 
\begin{align*}({\bsH_1+1})^{-1}&\bsB\bsA_-\xi=({\bsH_1+1})^{-1}\bsB(\bsH_0+1)^{1/2}{\bsA_-}{(\bsH_0+1)^{-1/2}}\xi\\
&={(\bsH_1+1)^{-1/2}}\cdot{(\bsH_1+1)^{-1/2}}(\bsH_0+1)^{1/2}\cdot\bsB{\bsA_-}{(\bsH_0+1)^{-1/2}}\xi\\
&\quad -({\bsH_1+1})^{-1}[(\bsH_0+1)^{1/2}, \bsB]{\bsA_-}{(\bsH_0+1)^{-1/2}}\xi.
\end{align*}
By Lemma \ref{big_estimate_on_norm_k=1}, the operator ${(\bsH_1+1)^{-1/2}}(\bsH_0+1)^{1/2}$ extends to a bounded operator, and by the assumption $[(\bsH_0+1)^{1/2}, \bsB]$ also extends to a bounded operator. Hence, the operator $({\bsH_1+1})^{-1}\bsB\bsA_-$ extends to a bounded operator and
\begin{align*}
\Big\|\overline{({\bsH_1+1})^{-1}\bsB\bsA_-}\Big\|&\leq \Big\|\overline{{(\bsH_1+1)^{1/2}}(\bsH_0+1)^{-1/2}}\Big\|\|\bsB\|+\|\delta_{\bsH_0}(\bsB)\|\\
&\leq \const\cdot(1+\|\delta_{\bsH_0}(\bsB)\|+\|\bsB\|+\|\bsB\|^2+\|\bsB'\|)\cdot\|\bsB\|\\
&\qquad+\|\delta_{\bsH_0}(\bsB)\|,
\end{align*}
where the latter inequality follows from Lemma \ref{big_estimate_on_norm_k=1}.

For the third term on the right hand side \eqref{big_estimate_on_norm_k=2}, we write
\begin{align*}
({\bsH_1+1})^{-1}\bsA_-\bsB\xi&={(\bsH_1+1)^{-1/2}}\cdot{(\bsH_1+1)^{-1/2}}(\bsH_0+1)^{1/2}\\
&\quad \times{(\bsH_0+1)^{-1/2}}{\bsA_-}\bsB\xi,
\end{align*}
 for $ \xi\in\dom(\bsH_0),$
and therefore, by Lemma \ref{big_estimate_on_norm_k=1} we conclude that $(\bsH_1+1)^{-1}\bsA_-\bsB$ also extends to a bounded operator. 

Thus, $({\bsH_1+1})^{-1}(\bsH_0+1)$ extends to a bounded operator and by \eqref{big_estimate_on_norm_k=2} we have 
\begin{align*}
\Big\|&\overline{({\bsH_1+1})^{-1}(\bsH_0+1)}\Big\|\leq 1+\Big\|\overline{({\bsH_1+1})^{-1}\bsB\bsA_-}\Big\|\\
&\qquad+\Big\|\overline{({\bsH_1+1})^{-1}\bsA_-\bsB}\Big\|+\Big\|({\bsH_1+1})^{-1}(\bsB^2-\bsB')\Big\|\\
&\leq Q(\|\delta_{\bsH_0}^j(\bsB)\|), \quad j=0,1.
\end{align*}
for some polynomial $Q$.

Suppose now that for some $k\geq 3$ the assertion holds for all $j\leq k-1$. Let us prove it for $j=k$. For $\xi\in \dom(\bsH_0^{k/2})$, using the resolvent identity \eqref{ssss} we write
\begin{align}\label{big_uniform_norm_k_into_two}
&{(\bsH_1+1)^{-k/2}}(\bsH_0+1)^{k/2}\xi={(\bsH_1+1)^{-(k-2)/2}}({\bsH_1+1})^{-1}(\bsH_0+1)^{k/2}\xi \nonumber\\
&={(\bsH_1+1)^{-(k-2)/2}}(\bsH_0+1)^{(k-2)/2}\xi\\
&\quad+
{(\bsH_1+1)^{-k/2}}(\bsB\bsA_-+\bsA_-\bsB+\bsB^2-\bsB')(\bsH_0+1)^{(k-2)/2}\xi\nonumber.
\end{align}
By the induction hypothesis, the operator ${(\bsH_1+1)^{-(k-2)/2}}(\bsH_0+1)^{(k-2)/2}$ extends to a bounded operator with the required estimate in the uniform norm. 

For the second term on the right hand side of \eqref{big_uniform_norm_k_into_two} equality \eqref{big_BH_0} implies that 
\begin{align*}
&{(\bsH_1+1)^{-k/2}}(\bsB\bsA_-+\bsA_-\bsB+\bsB^2-\bsB')(\bsH_0+1)^{(k-2)/2}\xi\\
=\sum_{j=0}^{k-1}&(-1)^jC_k^j{(\bsH_1+1)^{-k/2}}(\bsH_0+1)^{j/2}[(\bsH_0+1)^{1/2},\bsB]^{(j)}{\bsA_-}{(\bsH_0+1)^{-1/2}}\xi\\
+\sum_{j=0}^{k-2}&(-1)^jC_k^j{(\bsH_1+1)^{-k/2}}(\bsH_0+1)^{(j+1)/2}{\bsA_-}{(\bsH_0+1)^{-1/2}}[(\bsH_0+1)^{1/2},\bsB]^{(j)}\xi
\\
+\sum_{j=0}^{k-2}&(-1)^jC_k^j{(\bsH_1+1)^{-k/2}}(\bsH_0+1)^{j/2}[(\bsH_0+1)^{1/2},\bsB^2-\bsB']^{(j)}\xi\\
\end{align*}
By the induction hypothesis, for every $j=0,\dots, k-1$, it follows that the operator ${(\bsH_1+1)^{-k/2}}(\bsH_0+1)^{j/2}$ extends to a bounded operator. In addition,  the operators $[(\bsH_0+1)^{1/2},\bsB]^{(j)}$
and $ [(\bsH_0+1)^{1/2},\bsB']^{(j)}$ also extend to bounded operators by the assumption of the proposition. Hence, 
$${(\bsH_1+1)^{-k/2}}(\bsB\bsA_-+\bsA_-\bsB+\bsB^2-\bsB')(\bsH_0+1)^{(k-2)/2}$$
extends to a bounded operator and the required estimate in the uniform norm follows.
\end{proof}

\begin{proposition}\label{big_uniform_bounded} 
Let $\bsB,\bsB'\in\bigcap_{j=1}^{k-1}\dom(\delta_{\bsH_0}^j)$ for some $k\in\bbN$ and let $z\in \bbC\setminus\bbR$. Then 

\begin{enumerate}
\item The operator  $(\bsH_0-z)^{k/2}{(\bsH_{i}-z)^{-k/2}}$, \, $i=1,2$, is bounded.  
\item The operators $\overline{{(\bsH_{i,n}-z)^{-k/2}}(\bsH_0-z)^{k/2}},$  and $(\bsH_0-z)^{k/2}{(\bsH_{i,n}-z)^{-k/2}}$,\, $i=1,2$, are bounded. 
\item The sequences
$$\Big\{\overline{{(\bsH_{i,n}-z)^{-k/2}}(\bsH_0-z)^{k/2}}\Big\}_{n=1}^\infty, \quad\Big\{(\bsH_0-z)^{k/2}{(\bsH_{i,n}-z)^{-k/2}}\Big\}_{n=1}^\infty$$
are uniformly bounded.
\end{enumerate}

\end{proposition}

\begin{proof}Without loss of generality we can assume that $z=-1$. 

$(i).$ As the operators $(\bsH_0+1)^{k/2}$ and ${(\bsH_{i,n}+1)^{-k/2}}$ are self-adjoint, both of the operators ${(\bsH_{i,n}+1)^{-k/2}}$ and ${(\bsH_{i,n}+1)^{-k/2}}(\bsH_0+1)^{k/2}$ are densely defined, \cite[Theorem 4.19 (b)]{Weidmann} implies that 
\begin{align*}
(\bsH_0+1)^{k/2}{(\bsH_{i}+1)^{-k/2}}&=\Big({(\bsH_{i}+1)^{-k/2}}(\bsH_0+1)^{k/2}\Big)^*\\
&=\Big(\overline{{(\bsH_{i}+1)^{-k/2}}(\bsH_0+1)^{k/2}}\Big)^*\in\cB(L^2(\bbR;\cH)),
\end{align*}
where the last inclusion follows from  Proposition \ref{big_estimate_on_norm}.

$(ii).$ Since $\bsB,\bsB'\in\bigcap_{j=1}^{k-1}\dom(\delta_{\bsH_0}^j)$, $\bsB_n=\bsP_n\bsB\bsP_n, \, \bsB'_n=\bsP_n\bsB'\bsP_n,$ and $\bsP_n$ commutes with $\bsH_0$, we infer that $\bsB_n,\bsB'_n\in\bigcap_{j=1}^{k-1}\dom(\delta_{\bsH_0}^j)$. Therefore, applying   Proposition \ref{big_estimate_on_norm} and part (i) to the operators $\bsH_{i,n}$ and $\bsH_0$, we obtain the assertion.

$(iii).$ Note that for $j=1,\dots, k-1,$ we have 
$$\|\delta_{\bsH_0}^j(\bsB_n)\|\leq \|\delta_{\bsH_0}^j(\bsB)\|,\quad \|\delta_{\bsH_0}^j(\bsB_n')\|\leq \|\delta_{\bsH_0}^j(\bsB')\|.$$
Hence,  Proposition \ref{big_estimate_on_norm} applied to the operators $\bsH_{i,n}$ and $\bsH_0$ implies that for some polynomial $Q$ with positive coefficients, we have 
\begin{align*}
\Big\|\overline{{(\bsH_{i,n}+1)^{-k/2}}(\bsH_0+1)^{k/2}}\Big\|\leq \const Q(\|\delta_{\bsH_0}^j(\bsB_n)\|, \|\delta_{\bsH_0}^j(\bsB'_n)\|)\\
\leq \const Q(\|\delta_{\bsH_0}^j(\bsB)\|, \|\delta_{\bsH_0}^j(\bsB')\|)
,\quad j=0,\dots k-1,
\end{align*}
which together with the equality 
$$\Big\|(\bsH_0-z)^{k/2}{(\bsH_{i,n}-z)^{-k/2}}\Big\|=\Big\|\overline{{(\bsH_{i,n}+1)^{-k/2}}(\bsH_0+1)^{k/2}}\Big\|,$$
concludes the proof.
\end{proof}

\begin{corollary}\label{big_domains}
Let $\bsB,\bsB'\in\bigcap_{j=1}^{k-1}\dom(\delta_{\bsH_0}^j)$ for some $k\in\bbN$. Then 
$$\dom(\bsH_{i,n}^{k/2})=\dom(\bsH_i^{k/2})\subset\dom(\bsH_0^{k/2}),\quad i=1,2,\quad n\in\bbN.$$
\end{corollary}
\begin{proof}We prove only the equality $\dom(\bsH_0^{k/2})=\dom(\bsH_1^{k/2})$ since the others can be proved similarly. 

Let $\xi\in\dom(\bsH_1^{k/2})=\dom(\bsH_1+1)^{k/2}=\ran({(\bsH_1+1)^{-k/2}})$. Then there exists $\eta\in L^2(\bbR;\cH)$ such that $\xi={(\bsH_1+1)^{-k/2}}\eta$.
Since $\eta\in L^2(\bbR;\cH)$ and by Corollary \ref{big_uniform_bounded} (i) the operator $(\bsH_0+1)^{k/2}{(\bsH_{1}+1)^{-k/2}}$ is bounded, we have that 
$$(\bsH_0+1)^{k/2}\xi=(\bsH_0+1)^{k/2}{(\bsH_{1}+1)^{-k/2}}\eta\in\cH,$$
that is $\xi\in\dom(\bsH_0+1)^{k/2}=\dom(\bsH_0^{k/2})$.
\end{proof}

By Propositions \ref{big_estimate_on_norm} and  \ref{big_uniform_bounded}  the operators $\overline{{(\bsH_{2,n}-z )^{-\frac{k}2}}(\bsH_0-z)^{\frac{k}2}}$, $n\in\bbN$,  and $\overline{{(\bsH_{2}-z )^{-\frac{k}2}}(\bsH_0-z)^{\frac{k}2}}$ are bounded for every $z\in \bbC\setminus\bbR$.
The following proposition establishes the strong-operator convergence of the sequences  $\{\overline{{(\bsH_{2,n}-z )^{-\frac{k}2}}(\bsH_0-z)^{\frac{k}2}}\}_{n\in\bbN}$ to $\overline{{(\bsH_{2}-z )^{-\frac{k}2}}(\bsH_0-z)^{\frac{k}2}}$, which is required for the proof of the principal trace formula. 
The result here should be compared with \cite[Lemma 3.13 (ii)]{CGLS16}, where a much simpler case $k=2$ was treated.
\begin{proposition}\label{big_so_convergence}
Assume that $\bsB,\bsB'\in\bigcap_{j=1}^{k-1}\dom(\delta_{\bsH_0}^j)$ for some $k\in\bbN$  and let $z\in \bbC\setminus\bbR$. Then 
\begin{enumerate}
\item $\overline{{(\bsH_{2,n}-z )^{-\frac{k}2}}(\bsH_0-z)^{\frac{k}2}}$ converges to $\overline{{(\bsH_{2}-z )^{-\frac{k}2}}(\bsH_0-z)^{\frac{k}2}}$ 
in the strong operator topology.
\item $(\bsH_0-z)^{\frac{k}2}{(\bsH_{1,n} -z )^{-\frac{k}2}}$ converges to $(\bsH_0-z)^{\frac{k}2}{(\bsH_{1} -z )^{-\frac{k}2}}$
in the strong operator topology.
\end{enumerate}
\end{proposition}

\begin{proof}Without loss of generality we have $z=-1$. 

$(i).$ By Lemma \ref{approx_H_jn} we have that $\bsH_{2,n}\to\bsH_2$ in the strong resolvent sense. Therefore, \cite[Theorem VIII.20]{RS_book_I} implies that ${(\bsH_{2,n}+1 )^{-\frac{k}2}}\to {(\bsH_{2}+1 )^{-\frac{k}2}}$ in the strong operator topology. Hence, for every $\xi\in\dom(\bsH_0^{\frac{k}2})$ we have 
\begin{align*}
{(\bsH_{2,n}+1 )^{-\frac{k}2}}(\bsH_0+1)^{\frac{k}2}\xi\to {(\bsH_{2}+1 )^{-\frac{k}2}}(\bsH_0+1)^{\frac{k}2}\xi.
\end{align*}
Since $\dom(\bsH_0^{\frac{k}2})$ is a dense subset in $L^2(\bbR;\cH)$ and by Proposition \ref{big_uniform_bounded} (iii) the sequence $\{\overline{{(\bsH_{2,n}+1 )^{-\frac{k}2}}(\bsH_0+1)^{\frac{k}2}}\}_{n\in\bbN}$ is uniformly bounded, we infer that 
$$\overline{{(\bsH_{2,n}+1 )^{-\frac{k}2}}(\bsH_0+1)^{\frac{k}2}}\to \overline{{(\bsH_{2}+1 )^{-\frac{k}2}}(\bsH_0+1)^{\frac{k}2}}$$
in the strong operator topology. 

$(ii).$ 
By Corollary \ref{big_domains} we have that
$$\dom(\bsH_{1}+1)^{\frac{k}2}=\dom(\bsH_{1,n}+1)^{\frac{k}2}\subset \dom(\bsH_0+1)^{\frac{k}2},$$ and therefore both
${(\bsH_{1,n}+1 )^{-\frac{k}2}}\xi$ and ${(\bsH_{1}+1 )^{-\frac{k}2}}\xi$ lie in $\dom(\bsH_0+1)^{\frac{k}2}$ for every $\xi\in L^2(\bbR;\cH)$. 
The strong resolvent convergence $\bsH_{1,n}\to\bsH_1$ and   \cite[Theorem VIII.20]{RS_book_I}, imply that ${(\bsH_{1,n}+1 )^{-\frac{k}2}}\to{(\bsH_{1}+1 )^{-\frac{k}2}}$ in the strong operator topology. Hence, 
$$(\bsH_0+1)^{\frac{k}2}{(\bsH_{1,n} +1 )^{-\frac{k}2}}\xi\to (\bsH_0+1)^{\frac{k}2}{(\bsH_{1} +1 )^{-\frac{k}2}}\xi$$
for every $\xi\in L^2(\bbR;\cH)$, since the operator $(\bsH_0+1)^{\frac{k}2}$ is closed.
\end{proof}

\section{The approximation result for the pair $(\bsH_2, \bsH_1)$}\label{sec_hyp}
In this section we give the precise hypotheses that we impose for our results and expand on the details of what they mean. We also develop further our  approximation scheme that is essential for the proof of the principal trace formula in Section \ref{ch_principla trace formula} below.

\begin{hypothesis}\label{hyp_final} Assume $A_-$ is self-adjoint on $\dom(A_-) \subseteq \cH$ 
\begin{enumerate}
\item Suppose we have a family of bounded self-adjoint operators 
$\{B(t)\}_{t \in \bbR} \subset \cB(\cH)$, continuously differentiable with respect to $t$ in the uniform operator norm,  
such that 
$\|B'(\cdot)\|\in L^1(\bbR; dt)\cap L^\infty(\bbR; dt).  $   
\item Suppose that for some $p\in\bbN\cup\{0\}$, we have  
$$B'(t)(A_-+i)^{-p-1}\in\cL_1(\cH),\quad \int_\bbR \|B'(t)(A_-+i)^{-p-1}\|_1dt<\infty.$$
\item Let $m=\lceil \frac{p}2\rceil$.
Assume that for all $z<0$ we have that 
$$\bsB'(\bsH_0-z)^{-m-1}\in\cL_1(L^2(\bbR;\cH)).$$ 
\item $\bsB,\bsB'\in\bigcap_{j=1}^{2m-1}\dom(\delta_{\bsH_0}^j)$, where $\delta_{\bsH_0}$ is defined in \eqref{def_deriv_bsH_0}.
\end{enumerate}
In what follows we always take the smallest $p\in\bbN\cup\{0\}$ satisfying (iii).
\end{hypothesis}

%

Assuming now Hypothesis \ref{hyp_final}, we can prove our second key result, which guarantees that we can use approximation on the left-hand side of the principal trace formula. The proof of this result crucially uses the results obtained in Section \ref{sec_unifbound}.


\begin{theorem}\label{left-hand side_conver_in_L1}Assume Hypothesis \ref{hyp_final}. Let $z\in\bbC\setminus\bbR_+$. 
\begin{enumerate}
\item Both  
$(\bsH_2-z)^{-m}-(\bsH_1-z)^{-m}$
and $(\bsH_{2,n}-z)^{-m}-(\bsH_{1,n}-z)^{-m}$
are trace class.
\item We have $$\lim_{n\rightarrow \infty }\Big\|\big[(\bsH_{2,n}-z)^{-m}-(\bsH_{1,n}-z)^{-m}\big]-\big[(\bsH_2-z)^{-m}-(\bsH_1-z)^{-m}\big]\Big\|_1=0.$$
\end{enumerate}
\end{theorem}
\begin{proof}
$(i).$ 
Using again the resolvent identity and the elementary relation
\begin{equation*}
A^k - B^k = \sum_{j=1}^k A^{k-j} [A - B] B^{j-1}, \quad A, B \in \cB(\cH), 
\; k \in \bbN,
\end{equation*}
 we write
\begin{align*}
\left(\bsH_2 -z \right)^{-m}& - \left(\bsH_1-z \right)^{-m}\\
&=\sum_{j=1}^m{(\bsH_2 -z )^{-m+j}}({(\bsH_2 -z)^{-1} }-{(\bsH_1 -z)^{-1}}){(\bsH_1 -z )^{-j+1}}\\
&=-2\sum_{j=1}^m{(\bsH_2 -z )^{-m+j-1}}\bsB'{(\bsH_1 -z )^{-j}}.
\end{align*}
Thus
\begin{align}\label{conv_H_m_trace}
&\left(\bsH_2 -z \right)^{-m} - \left(\bsH_1-z \right)^{-m}=-2\sum_{j=1}^m\overline{{(\bsH_2-z )^{-m+j-1}}(\bsH_0-z)^{m-j+1}}\nonumber\\
&\qquad\times{(\bsH_0-z)^{-m+j-1}}\bsB'{(\bsH_0-z)^{-j}}\times(\bsH_0-z)^{j}{(\bsH_1 -z )^{-j}}
\end{align}

 We note that 
by the three lines theorem (see also \cite[Theorem 3.2]{GLST15}), Hypothesis \ref{hyp_final} (iii) implies that 
\begin{equation}\label{diff_powers_resolvents_big}
(\bsH_0-z)^{-m+j-1}\bsB'(\bsH_0-z)^{-j}\in\cL_1(L^2(\bbR;\cH))
\end{equation}
for all $j=1,\dots,m.$
%
Since, in addition, by Proposition \ref{big_uniform_bounded} (i) and (ii),  the operators  
$\overline{{(\bsH_2-z )^{-m+j-1}}(\bsH_0-z)^{m-j+1}}$ and $(\bsH_0-z)^{j}{(\bsH_1 -z )^{-j}}$ are bounded, we infer that 
$\left(\bsH_2 -z \right)^{-m} - \left(\bsH_1-z \right)^{-m}\in\cL_1(L^2(\bbR;\cH)).$

Arguing similarly, one may obtain that 
\begin{align}\label{conv_H_n_m_trace}
&(\bsH_{2,n} -z )^{-m} - (\bsH_{1,n}-z )^{-m}
\nonumber\\
&=-2\sum_{j=1}^m{(\bsH_{2,n} -z )^{-m+j-1}}\bsP_n\bsB'\bsP_n{(\bsH_{1,n} -z )^{-j-1}}\nonumber\\
&=-2\sum_{j=1}^m\overline{{(\bsH_{2,n}-z )^{-m+j-1}}(\bsH_0-z)^{m-j+1}}\\
&\quad\times\bsP_n{(\bsH_0-z)^{-m+j-1}}\bsB'{(\bsH_0-z)^{-j}}\bsP_n\times(\bsH_0-z)^{j}{(\bsH_{1,n} -z )^{-j}}.\nonumber
\end{align}
Referring to Proposition \ref{big_uniform_bounded}
we have that 
$\left(\bsH_{2,n} -z \right)^{-m} - \left(\bsH_{1,n}-z \right)^{-m}\in\cL_1(L^2(\bbR;\cH)).$

$(ii).$ Using decompositions \eqref{conv_H_m_trace} and \eqref{conv_H_n_m_trace} we see that it is sufficient to prove the convergence of each term separately. 

By \eqref{diff_powers_resolvents_big}, the operator 
${(\bsH_0-z)^{-m+j-1}}\bsB'{(\bsH_0-z)^{-j}}\in{\cL_1(L^2(\bbR;\cH))}$ for all $j=1,\dots,m$, and therefore, by Lemma \ref{so_inLp} we have that 
$$\bsP_n{(\bsH_0-z)^{-m+j-1}}\bsB'{(\bsH_0-z)^{-j}}\bsP_n\xrightarrow{\|\cdot\|_1}{(\bsH_0-z)^{-m+j-1}}\bsB'{(\bsH_0-z)^{-j}}.$$
In addition, by Proposition \ref{big_so_convergence} we have 
$$\overline{{(\bsH_{2,n}-z )^{-m+j-1}}(\bsH_0-z)^{m-j+1}}\to \overline{{(\bsH_{2}-z )^{-m+j-1}}(\bsH_0-z)^{m-j+1}}$$
and 
$$(\bsH_0-z)^{j}{(\bsH_{1,n} -z )^{-j}}\to (\bsH_0-z)^{j}{(\bsH_{1} -z )^{-j}}, \quad j=1,\dots, m,$$
in the strong operator topology. Thus,  another appeal  to 
Lemma \ref{so_inLp}  
 completes the proof. 
\end{proof}

In the conclusion of the present section, we discuss some details of our main assumption, Hypothesis \ref{hyp_final}, for the special path $\{B(t)\}_{t\in\bbR}$ defined as follows. Suppose that a positive function $\theta$ on $\bbR$ satisfies
\begin{align}\label{theta}
\begin{split}
\theta\in C_b^\infty(\bbR),& \quad \theta'\in L_1(\bbR),\\
\lim_{t\to-\infty}\theta(t)=0,&\quad \lim_{t\to+\infty}\theta(t)=1.
\end{split}
\end{align}
and assume that $B_+$ is a $p$-relative perturbation of $A_-$.
Introduce then the family $\{B(t)\}_{t\in\bbR}$ given by 
\begin{equation}\label{Bt_special}
B(t)=\theta(t)B_+.
\end{equation}

Since $\theta'\in L_1(\bbR)$ and $B'(t)(A_-+i)^{-p-1}=\theta'(t)B_+(A_-+i)^{-p-1}$, it follows that the assumptions
 (i) and (ii) of Hypothesis \ref{hyp_initial} are satisfied. Furthermore, an argument similar to the proof of \cite[Proposition 2.2]{CGK16} guarantees that 
$$\bsB'(\bsH_0-z)^{-m-1}\in\cL_1(L^2(\bbR;\cH)),$$
that is, assumption  (iii)  of Hypothesis \ref{hyp_final} is satisfied.

%
%

Thus, for the special type of family $\{B(t)\}_{t\in\bbR}=\{\theta(t)B_+\}_{t\in\bbR}$  assumptions  (ii), (iii) and (iv) of Hypothesis \ref{hyp_final} are automatically guaranteed by the assumption \eqref{theta} { on $\theta$} and the fact that $B_+$ is a $p$-relative trace class perturbation of $A_-$.

Next, we discuss Hypothesis \ref{hyp_final} (iv). By definition of $\delta_{\bsH_0}$ (see \eqref{def_deriv_bsH_0}) to check Hypothesis~\ref{hyp_final}~(iv) we have to consider repeated commutators with $(1+\bsH_0)^{1/2}$. However, in general, it is hard to work with these commutators. Therefore, we use below a different type of commutator argument, in which 
the commutators are more manageable. 

We now follow a method introduced in  \cite[Section 1.3]{CGRS_memoirs} using
the operator defined by
$\bsL_{\bsH_0}^k(\bsT):=\overline{(1+\bsH_0)^{-k/2}[\bsH_0,T]^{(k)}}$ whose domain is 
\begin{align*}
\dom(\bsL_{\bsH_0}^k)=\{ \bsT\in\cB(L^2(\bbR;\cH)): \bsT\dom(\bsH_0^j)\subset \dom(\bsH_0^j), \, j=1,\dots, k\\
\text{ 
and the operator } (1+\bsH_0)^{-k/2}[\bsH_0,\bsT]^{(k)} \text{ defined on }\dom(\bsH_0^{k})\\
\text{ extends to a bounded operator on } L^2(\bbR;\cH)\}.
\end{align*}
%

The following result follows from the proof of \cite[Lemma 1.29]{CGRS_memoirs}.
\begin{lemma}\label{dom_delta_and_L}
If \, $\bsT\in\bigcap_{j=1}^{2k}\dom(\bsL_{\bsH_0}^j)$ for some $k\in\bbN$, then $\bsT\in\bigcap_{j=1}^k\dom(\delta_{\bsH_0}^j)$.
\end{lemma}
%

Next, we want to reduce the commutators with $\bsH_0$ to commutators with $\bsA_-^2$. To this end, for a self-adjoint operator $A$ on $\cH$  we 
introduce the operator
\begin{equation}\label{def_L}
L_{A^2}^k(T)=\overline{(1+A^2)^{-k/2}[A^2,T]^{(k)}}
\end{equation} 
 with domain
\begin{align*}
\dom(L_{A^2}^k)=\{ T\in\cB(\cH): T\dom(A^j)\subset \dom(A^j), \, j=1,\dots, 2k\\
\text{ 
and the operator } (1+A^2)^{-k/2}[A^2,T]^{(k)} \text{ defined on }\dom(A^{2k})\\
\text{ extends to a bounded operator on } \cH\}.
\end{align*}

\begin{proposition}\label{prop_hyp_3}
Let $\{B(t)\}_{t\in\bbR}$ be as in \eqref{Bt_special} with $\theta$ satisfying \eqref{theta}. If $B_+\in \bigcap_{j=1}^{k}\dom(L_{A_-^2}^j),$ for some $k\in\bbN$, then $\bsB,\bsB'\in\bigcap_{j=1}^{k}\dom(\bsL_{\bsH_0}^j).$
\end{proposition}
\begin{proof}We prove the assertion for $\bsB$ only, as the assertion for $\bsB'$ can be proved similarly. 
First, identifying the Hilbert spaces $L^2(\bbR;\cH)$ and $L^2(\bbR)\otimes H$, we have  
$$\bsH_0=\frac{d^2}{dt^2}+\bsA_-^2=\frac{d^2}{dt^2}\otimes 1+1\otimes A_-^2.$$
Therefore, since $\theta\in C_b^\infty(\bbR)$ and $B_+\dom(A_-^j)\subset \dom(A_-^j), \, j=1,\dots, 2k,$ it follows that the operator $\bsB=M_\theta\otimes B_+$ leaves $\dom(\bsH_0^j), j=1,\dots, k$, invariant. 

Furthermore, on $\dom(\bsH_0)$ we have 
\begin{align*}
[\bsH_0,\bsB]&=[\frac{d^2}{dt^2}\otimes 1, M_\theta\otimes B_+]+[1\otimes A_-^2,M_\theta\otimes B_+]\\
&=[\frac{d^2}{dt^2}, M_\theta]\otimes B_++M_\theta\otimes [A_-^2,\otimes B_+].
\end{align*}
Hence, on $\dom (\bsH_0^k)$ we have 
$[\bsH_0,M_\theta\otimes B_+]^{(k)}=\sum_{l=0}^k [\frac{d^2}{dt^2}, M_{\theta}]^{(l)}\otimes [A_-^2,B_+]^{(k-l)}.$

%

The operator 
$\bsC:=\Big(1-\frac{d^2}{dt^2}\Big)^{\frac{l}{2}}\Big(1+\bsA_-^2\Big)^{\frac{k-l}{2}}(1+\bsH_0)^{-k/2}$ is bounded for any $l=0,\dots, k$ By Lemma \ref{lem_A_H_0_bounded}.
Hence, on $\dom (\bsH_0^k)$ we have 
\begin{align*}
\bsL_{\bsH_0}^k(\bsB)&=(1+\bsH_0)^{-k/2}[\bsH_0,M_\theta\otimes B_+]^{(k)}\\
&=(1+\bsH_0)^{-k/2}\sum_{l=0}^k [\frac{d^2}{dt^2}, M_{\theta}]^{(l)}\otimes [A_-^2,B_+]^{(k-l)}\\
&=\sum_{l=0}^k \bsC\cdot 
(1-\frac{d^2}{dt^2})^{-\frac{l}{2}}[\frac{d^2}{dt^2}, M_{\theta}]^{(l)} \otimes (1+A_-^2)^{-\frac{k-l}2}[A_-^2,B_+]^{(k-l)}.
\end{align*}
It follows from inclusion \eqref{inc_L_laplace} below (for $n=1$) that $(1-\frac{d^2}{dt^2})^{-\frac{l}{2}}[\frac{d^2}{dt^2}, M_{\theta}]^{(l)}$ extends to a bounded operator on $L^2(\bbR)$ for any $l=0,\dots,k$. By assumption the operator $(1+A_-^2)^{-{k-l}/2}[A_-^2,B_+]^{(k-l)}$ extends to a bounded operator on $\cH$. Therefore, $\bsL_{\bsH_0}^k(\bsB)$ also extends to a bounded operator on $L^2(\bbR;\cH)$, as required. 
\end{proof}

We now formulate the Hypothesis \ref{hyp_final} for the special case when the family $\{B(t)\}_{t\in\bbR}$ is given by $\{\theta(t)B_+\}_{t\in\bbR}$. 
\begin{hypothesis}\label{hyp_final_special}
\begin{enumerate}
\item Assume that $A_-$ is self-adjoint on $\dom(A_-) \subseteq \cH$ 
and let $\theta$ satisfy \eqref{theta}.
\item Suppose that an operator $B_+$ is a $p$-relative trace-class perturbation of $A_-$ for some $p\in\bbN$, that is 
$B_+(A_-+i)^{-p-1}\in\cL_1(\cH).$
\item Assume also that $B_+\in \bigcap_{j=1}^{2p}\dom(L_{A_-^2}^j),$ where the mapping $L_{A_-^2}^j$ is defined by \eqref{def_L}. 
\end{enumerate}

\end{hypothesis}

For convenience, we state the following
\begin{proposition} \label{prop_hyp_special}
For the special case when $\{B(t)\}_{t\in\bbR}=\{\theta(t)B_+\}_{t\in\bbR}$, Hypothesis \ref{hyp_final_special} guarantees that Hypothesis \ref{hyp_final} is satisfied. 
\end{proposition}

\section{The principal trace formula}\label{ch_principla trace formula}

In this section we prove the fundamental result of the present paper, the principal trace formula \eqref{ch_principla trace formula_intro_formula},
 which states that
\begin{equation*}
 \tr\Big(e^{-t\bsH_2}-e^{-t\bsH_1}\Big)=-\Big(\frac{t}{\pi}\Big)^{1/2}\int_0^1\tr\Big(e^{-tA_s^2}(A_+-A_-)\Big)ds,\quad t>0,
 \end{equation*}
where $A_s=A_-+s(A_+-A_-), s\in[0,1]$ is the straight line path joining $A_-$ and $A_+$ (see Theorem \ref{thm_principla trace formula}).
 As mentioned in the introduction, the principal trace formula allows us to establish  further results for the Witten index (see Section \ref{ch_WI} below). In fact the results of Sections  \ref{ch_WI} are (almost immediate) corollaries of the principal trace formula.  

Our approach to the proof of the principal trace formula relies on an approximation method in which we use results already known for the path $\{A_n(t)\}_{t\in\bbR}$ of reduced operators
$A_n(t)=A_-+P_nB(t)P_n,$ where, as before, $P_n=\chi_{[-n,n]}(A_-)$. We first recall a result from \cite{CGPST_Witten}. The notation $\erf$ stands for the error function 
\begin{equation}\label{erf_def}
\erf(x) = \f{2}{\pi^{1/2}} \int_0^x  e^{- y^2}\, dy, \quad x \in \bbR.
\end{equation}

\begin{proposition}\cite[Example B.6 (ii) and Theorem B.5]{CGPST_Witten}
\label{prop_principla trace formula_for_reduced}
For the path $\{A_n(t)\}_{t\in\bbR}$ of reduced operators we have that $$e^{-t\bsH_{2,n}}-e^{-t\bsH_{1,n}}\in\cL_1(L^2(\bbR;\cH)),\quad \erf(t^{1/2}A_{+,n})-\erf(t^{1/2}A_-)\in\cL_1(\cH)$$ and the equation
\begin{equation}\label{principla trace formula_for_reduced}
\tr\Big(e^{-t\bsH_{2,n}}-e^{-t\bsH_{1,n}}\Big)=-\frac12\tr\Big(\erf(t^{1/2}A_{+,n})-\erf(t^{1/2}A_-)\Big),
\end{equation}
holds.
\end{proposition}
%
%

\begin{remark}
We note that under the assumption of Hypothesis \ref{hyp_final}, one can not pass to the limit as $n\to\infty$ in \eqref{principla trace formula_for_reduced}, in general.
 Indeed, consider the case, 
when the operator $A_-$ is the two-dimensional Dirac operator (see Section  \ref{ch_examples}) and the perturbed operator $A_+$ is given by 
$A_+=A_-+1\otimes M_f,\quad f\in S(\bbR).$ 
The family $\{\theta(t) \otimes M_f\}_{t\in\bbR}$, with $\theta$ given by \eqref{theta}, satisfies Hypothesis \ref{hyp_final} (see Section \ref{ch_examples}). However, by \cite[Theorem 1.2 (i)]{LSVZ_function_g} we have that the operator 
$$\erf(t^{1/2}A_{+})-\erf(t^{1/2}A_-)$$
is not a trace-class operator. 
\hfill $\diamond$
\end{remark}

To overcome the obstacle mentioned in the above remark we firstly 
rewrite the principal trace formula obtained in Proposition \ref{prop_principla trace formula_for_reduced} for the reduced operators. 
 
 \begin{lemma}\label{lem_principla trace formula_reduced}
 For the path $\{A_n(t)\}_{t\in\bbR}$ of reduced operators we have
 \begin{equation*}
\tr\Big(e^{-t\bsH_{2,n}}-e^{-t\bsH_{1,n}}\Big)=-\Big(\frac{t}{\pi}\Big)^{1/2}\int_0^1\tr\Big(e^{-tA_{s,n}^2}(A_{+,n}-A_-)\Big)ds,
\end{equation*}
where $A_{s,n}=A_-+sP_nB_+P_n$, \, $s\in[0,1]$.
 \end{lemma}
 
\begin{proof}
By Proposition \ref{prop_principla trace formula_for_reduced} we have 
$$\tr\Big(e^{-t\bsH_{2,n}}-e^{-t\bsH_{1,n}}\Big)=-\frac12\tr\Big(\erf(t^{1/2}A_{+,n})-\erf(t^{1/2}A_-)\Big).$$

Since the operator $B_{+,n}=A_{+,n}-A_-$ is a trace-class operator (see \eqref{B_+,n_trace_class}), it follows that the path $B_{s,n}=s(A_{+,n}-A_-)$ is a $C^1$-path of trace-class operators. Applying now  Daletski-Krein formula \cite{BS_DOI_III} (see also  \cite{Simon_PAMS}) for this path we obtain that 
\begin{equation}\label{FTC_erf}
\frac12\tr\Big(\erf(t^{1/2}A_{+,n})-\erf(t^{1/2}A_-)\Big)=\Big(\frac{t}{\pi}\Big)^{1/2}\int_0^1\tr\Big(e^{-tA_{s,n}^2}(A_{+,n}-A_-)\Big)ds.
\end{equation}
Hence, 
\begin{equation*}
\tr\Big(e^{-t\bsH_{2,n}}-e^{-t\bsH_{1,n}}\Big)=-\Big(\frac{t}{\pi}\Big)^{1/2}\int_0^1\tr\Big(e^{-tA_{s,n}^2}(A_{+,n}-A_-)\Big)ds,
\end{equation*}
as required.
\end{proof}

We now ready to prove the principal trace formula in its heat kernel version, which is the main result of this paper.
\begin{theorem}[The principal trace formula]\label{thm_principla trace formula}Assume Hypothesis \ref{hyp_final}. 
Let $A_s=A_-+s(A_+-A_-), s\in[0,1],$ be the straight line path joining $A_-$ and $A_+$.
 Then for all $t>0$, we have 
$$\tr\Big(e^{-t\bsH_2}-e^{-t\bsH_1}\Big)=-\Big(\frac{t}{\pi}\Big)^{1/2}\int_0^1\tr\Big(e^{-tA_s^2}(A_+-A_-)\Big)ds.$$
\end{theorem}
\begin{proof}
By Lemma \ref{lem_principla trace formula_reduced} we have 
\begin{equation}\label{principla trace formula_semigroup_reduced_2}
\tr\Big(e^{-t\bsH_{2,n}}-e^{-t\bsH_{1,n}}\Big)=-\Big(\frac{t}{\pi}\Big)^{1/2}\int_0^1\tr\Big(e^{-tA_{s,n}^2}(A_{+,n}-A_-)\Big)ds.
\end{equation}
We now pass to the limit as $n\to\infty$. 

For the left hand side of \eqref{principla trace formula_semigroup_reduced_2} we firstly note that Theorem \ref{left-hand side_conver_in_L1} guarantees that the assumptions of Theorem \ref{DOI_converge_pointwise} are satisfied with $A_n=\bsH_{2,n}, \, A=\bsH_2,\, B_n=\bsH_{1,n}, \, B=\bsH_1$. Hence,
Theorem \ref{DOI_converge_pointwise} implies that 
\begin{equation}\label{principla trace formula_semigroup_left-hand side}
\lim_{n\to\infty}\tr\Big(e^{-t\bsH_{2,n}}-e^{-t\bsH_{1,n}}\Big)=\tr\Big(e^{-t\bsH_{2}}-e^{-t\bsH_{1}}\Big).
\end{equation}
For the right hand side of \eqref{principla trace formula_semigroup_reduced_2} by Proposition \ref{prop_conv_integral_exp} we have 
\begin{equation*}
\lim_{n\to\infty}\int_0^1\tr\Big(e^{-tA_{s,n}^2}(A_{+,n}-A_-)\Big)ds=\int_0^1\tr\Big(e^{-tA_{s}^2}(A_{+}-A_-)\Big)ds.
\end{equation*}
Thus, \eqref{principla trace formula_semigroup_reduced_2} and \eqref{principla trace formula_semigroup_left-hand side} imply that 
\begin{equation*}
\tr\Big(e^{-t\bsH_{2}}-e^{-t\bsH_{1}}\Big)=-\Big(\frac{t}{\pi}\Big)^{1/2}\int_0^1\tr\Big(e^{-tA_{s}^2}(A_{+}-A_-)\Big)ds,
\end{equation*}
which concludes the proof. 
\end{proof}

\begin{remark}
Using a similar argument one can prove the following 
$$\tr\Big((\bsH_2-z)^{-k}-(\bsH_1-z)^{-k}\Big)=-\frac{(2k-1)!!}{2^k(k-1)!} \int_0^1\tr\big((A_s^2-z)^{-\frac12-k}(A_{+}-A_{-})\big)ds,$$
which gives an alternative proof of the resolvent version of the principal trace formula proved in \cite{CGK16}.
The resolvent version of the formula is known to be connected to cyclic homology (see references in \cite{CGK16}).

Our technique may be used to prove trace formulas for a wide class of functions thus going beyond resolvents and heat kernels. However, at this point it is not clear that there is a use for formulas for a  general class of functions, and so we omit this refinement.  \hfill $\diamond$
\end{remark}


\section{The Witten index and the spectral shift function}\label{ch_WI}

In this section  we move on to the implications for spectral shift functions of the principal trace formula, Theorem \ref{thm_WI=ssf}, which provides a relation between the Witten index of the operator $\bsD_\bsA^{}$  and the spectral shift function $\xi(\cdot;A_+,A_-)$. We follow the detailed treatment in \cite{CGPST_Witten}. 
What is new here compared with earlier work is that the results established in  this paper  
 enable us to remove the `relatively trace class perturbation assumption' in \cite{CGPST_Witten}  as well 
as the  `relatively Hilbert-Schmidt class perturbation assumption' in \cite{CGLS16}. This extension means that the old difficulty, that the only examples for
which the Witten index was defined were low dimensional,  is removed.  
In Section \ref{ch_examples} we show that our Hypothesis \ref{hyp_initial} permits consideration of differential operators (in particular, Dirac operators) in any dimension 
uniformly.

  We refer the reader to \cite[Chapter 8]{Yafaev_general} for the necessary background on the spectral shift function and its properties.
We start with introduction of the spectral shift functions $\xi(\cdot; A_+,A_-)$ and $\xi(\cdot; \bsH_2,\bsH_1)$. 

By Theorem \ref{left-hand side_conver_in_L1} we have 
$(\bsH_2 - z)^{-m} - (\bsH_1 - z)^{-m} \in 
\cL_1\big(L^2(\bbR; \cH)\big),$
for all $z\in\bbC\setminus \bbR$. Since the operators $\bsH_2, \bsH_1$ are non-negative, the function $\lambda\to (\lambda-z)^{-m}$ is monotone on the spectra of $\bsH_j$, $j=1,2$, for any $z<0$. Hence, it follows from  \cite[Section 8.9]{Yafaev_general} that there is a  spectral shift function 
$\xi(\, \cdot \, ; \bsH_2, \bsH_1)$ for the pair $(\bsH_2, \bsH_1)$ that satisfies
\begin{equation}
\xi(\, \cdot \, ; \bsH_2, \bsH_1) \in L^1\big(\bbR; (|\lambda|^{m + 1} + 1)^{-1} d\lambda\big).
\end{equation} 
Since $\bsH_j\geq 0$, $j=1,2$,  $\xi(\,\cdot\,; \bsH_2,\bsH_1)$ may be specified uniquely
by requiring that
\begin{equation}\label{eq_fixed_ssf_big}
\xi(\lambda; \bsH_2,\bsH_1) = 0, \quad \lambda < 0.  
\end{equation}
In addition, the following trace formula  holds:
\begin{equation}\label{trace_formula_left-hand side}
\tr\big(f(\bsH_2) - f(\bsH)\big)
= \int_{[0, \infty)} f'(\lambda)\xi(\lambda; \bsH_2, \bsH_1)\, d\lambda, \quad 
f \in S(\bbR).
\end{equation} 


We introduce now the spectral shift function for the pair $(A_+,A_-)$ using  \cite{Ya05}.
We firstly introduce the notation 
$p_0=2\lfloor \frac{p}2\rfloor+1,$
so that $p_0$ is always odd.
By Theorem \ref{thm_trace_class_approx} we have that $(A_+-z)^{-p_0}-(A_--z)^{-p_0}\in\cL_1(\cH)$, and therefore, by \cite[Theorem 2.2]{Ya05} there exists a function 
\begin{equation}\label{summability_xi_right-hand side}
 \xi(\cdot; A_+,A_-)
  \in L^1 (\mathbb R; (1 + \left| \lambda \right|)^{-p_0-1}\,  d
  \lambda) 
\end{equation} such that 
  \begin{equation}
    \label{trace_formula_right-hand side}
    \mathrm {tr}\, \left( f(A_+) - f(A_-)
    \right) = \int_{\mathbb R} f'(\lambda) \cdot \xi(\lambda; A_+ ,A_-)\, 
    d\lambda,\quad f\in S(\bbR).
  \end{equation}
%
%

However the spectral shift function $\xi(\cdot; A_+, A_-)$ introduced above is not unique,  in general, and therefore we have to fix one particular spectral shift function, which satisfies \eqref{summability_xi_right-hand side} and \eqref{trace_formula_right-hand side}. We do this using \cite{CGLNPS16}.

\begin{theorem}\label{fixing_ssf_A_-} Assume Hypothesis \ref{hyp_final}. There exists a unique spectral shift function $\xi(\cdot; A_+,A_-)$ such that
\begin{equation}\label{fixing_ssf_implic}
\xi(\cdot; A_+,A_-)=\lim_{n\to\infty}\xi(\cdot; A_{+,n},A_-)
\end{equation} 
in  $L^1(\bbR;(|\nu|^{p_0+1}+1)^{-1}d\nu)$.
\end{theorem}
\begin{proof}

Introduce the path $\{A_+(s)\}_{s \in [0,1]}$ by setting
\begin{align*}
& A_+(s) = A_- + \hat P_s B_+\hat P_s, 
\quad \dom(A_+(s)) = \dom(A_-),   \quad s \in [0,1],    \\
& \hat P_s=\chi_{[-\frac1{1-s},\frac1{1-s}]}(A_-), s \in [0,1), \quad \hat P_1=I,
\end{align*}
in particular, 
\begin{equation}\label{A_+(0)}
A_+(0) = A_{+,1} (\text{ see } \eqref{Ant} ),\quad\, A_+(1) = A_+. 
\end{equation} 

Since $B_+=A_+-A_-$ is a $p$-relative trace-class perturbation of $A_-$ (see \eqref{B_+_p-relative}), inclusion \eqref{B_+,n_trace_class} implies  that $P_nB_+P_n\in\cL_1(\cH)$, and therefore $A_+(s)-A_-=\hat P_sB_+\hat P_s\in\cL_1(\cH)$ for any $s<1$. Hence  there exists 
a unique spectral shift function $\xi(\cdot; A_+(s),A_-)$ for the pair $(A_+(s),A_-)$, $s<1$, satisfying, in particular, 
$\xi(\cdot; A_+(s),A_-)\in L^1(\bbR).$

Moreover, in complete analogy to Theorem \ref{thm_trace_class_approx}, the family 
$A_+(s)$ depends continuously on $s \in [0,1]$ with respect to the family of pseudometrics $d_{p_0,z}(\cdot,\cdot)$
\begin{equation*}
d_{p_0,z}(A,A') = \big\|(A - i I)^{-p_0} - (A' - i I)^{-p_0}\big\|_{\cL_1(\cH)}  
\end{equation*}
for $A, A'$ in the set of self-adjoint operators which are $p_0$-resolvent comparable with 
respect to $A_-$ (equivalently, $A_+$), that is, $A, A'$ satisfy 
for all $\zeta \in i\bbR\setminus\{0\}$, 
\begin{equation*}
\big[(A - \zeta I)^{-p_0} - (A_- - \zeta I)^{-p_0}\big], 
\big[(A' - \zeta I)^{-p_0} - (A_- - \zeta I)^{-p_0}\big]  \in \cL_1(\cH).  
\end{equation*}
Thus, the hypotheses of \cite[Theorem 4.7]{CGLNPS16} are satisfied and hence we conclude that there exists a unique spectral shift function 
$\xi(\, \cdot \,; A_+(s), A_-)$ for the pair $(A_+(s), A_-)$ depending continuously on the parameter $s \in [0,1]$ in 
the space $L^1\big(\bbR; (|\nu|^{p_0+1} +1)^{-1} d\nu\big)$, satisfying  
$\xi(\, \cdot \,; A_+(0), A_-) = \xi(\, \cdot \,; A_{+,1}, A_-)$.

Taking $s=\frac{n-1}{n}$ we obtain  
$$\xi(\, \cdot \, ; A_+, A_-) 
\stackrel{\eqref{A_+(0)}}{=} \xi(\, \cdot \, ; A_+(1), A_-) 
= \lim_{s \uparrow 1}\xi(\, \cdot \, ; A_+(s), A_-)
= \lim_{n \to \infty} \xi(\, \cdot \, ; A_{+,n}, A_-).$$
 Therefore,  $\{\xi(\, \cdot \, ; A_{+,n}, A_-)\}_{n \in \bbN}$  converges pointwise a.e. to $\xi(\cdot; A_+, A_-)$ as $n \to \infty$. 
Since every $\xi(\, \cdot \, ; A_{+,n}, A_-)$, \, $n \in \bbN$, is uniquely defined, we  can fix uniquely the spectral shift function $\xi(\, \cdot \, ; A_+, A_-)$ satisfying conditions \eqref{fixing_ssf_implic}. \end{proof}
\begin{remark}
Since every $\xi(\, \cdot \, ; A_{+,n}, A_-),n \in \bbN,$ is uniquely defined, Theorem \ref{fixing_ssf_A_-} implies that we can fix uniquely the spectral shift function $\xi(\, \cdot \, ; A_+, A_-)$ satisfying equation \eqref{fixing_ssf_implic}. We adopt this method of fixing the spectral shift function 
for the remainder of this paper.
\hfill $\diamond$
\end{remark}
%
%

Now we move on to the proof of Pushnitski's formula, which establishes a relation between the spectral shift functions $(A_+,A_-)$ and $(\bsH_2,\bsH_1)$.

\begin{theorem}\label{thm_Push}
Assume Hypothesis \ref{hyp_final}. Let $\xi(\cdot;A_+,A_-)$ be the spectral shift function for the pair $(A_+,A_-)$ fixed in \eqref{fixing_ssf_implic} and let $\xi(\cdot; \bsH_2, \bsH_1)$ be the  spectral shift function for the pair $(\bsH_2,\bsH_1)$ fixed by equality \eqref{eq_fixed_ssf_big}. Then for a.e. $\lambda>0$ we have 
\begin{equation}\label{eq_Push}
\xi(\lambda; \bsH_2, \bsH_1)=\frac{1}{\pi}\int_{-\lambda^{1/2}}^{\lambda^{1/2}}
\frac{\xi(\nu; A_+,A_-)\, d\nu}{(\lambda-\nu^2)^{1/2}}
\end{equation}
with a convergent Lebesgue integral on the right-hand side of \eqref{eq_Push}. 
\end{theorem}
\begin{proof}
By the principle trace formula for the semigroup difference (see Theorem \ref{thm_principla trace formula}) we have that 
$$\tr(e^{-t\bsH_2}-e^{-t\bsH_1})=-\Big(\frac{t}{\pi}\Big)^{1/2}\int_0^1\tr\big(e^{-tA_s^2}(A_+-A_-)\big)ds$$
for all $t>0$. 
On the right hand side of this formula using \eqref{conv_right-hand side_principla trace formula} and \eqref{FTC_erf} we have that 
\begin{align}
\begin{split}
\Big(\frac{t}{\pi}\Big)^{1/2}\int_0^1&\tr(e^{-tA_s^2}(A_+-A_-))ds\\
&\stackrel{\eqref{conv_right-hand side_principla trace formula}}{=}\lim_{n\to\infty}\Big(\frac{t}{\pi}\Big)^{1/2}\int_0^1\tr(e^{-tA_{s,n}^2}(A_{+,n}-A_-))ds
\\
&\stackrel{\eqref{FTC_erf}}{=}\lim_{n\to\infty}\frac12\tr(\erf(t^{1/2}A_{+,n})-\erf(t^{1/2}A_-))
\end{split}
\label{principla trace formula_right-hand side_to_ssf}
\end{align}
By Krein's trace formula \cite[Theorem 8.3.3]{Yafaev_general} and the definition of the error function \eqref{erf_def} it follows that 
\begin{equation}\label{Krein_small_erf}
\frac12\tr\big(\erf(t^{1/2}A_{+,n})-\erf(t^{1/2}A_-)\big)=\Big(\frac{t}{\pi}\Big)^{1/2}\int_\bbR e^{-ts^2}\xi(s;A_{+,n},A_-)ds.
\end{equation}
Furthermore, referring to Theorem \ref{fixing_ssf_A_-}
we obtain 
\begin{equation}\label{approx_ssf_erf}
\lim_{n\to\infty}\Big(\frac{t}{\pi}\Big)^{1/2}\int_\bbR e^{-ts^2}\xi(s; A_{+,n},A_-)ds=\Big(\frac{t}{\pi}\Big)^{1/2}\int_\bbR e^{-ts^2}\xi(s;A_{+},A_-)ds.
\end{equation}

Thus, combining \eqref{principla trace formula_right-hand side_to_ssf}, \eqref{Krein_small_erf} and \eqref{approx_ssf_erf} we conclude that the right-hand side of the principal trace formula can be written as 
\begin{equation*}
\Big(\frac{t}{\pi}\Big)^{1/2}\int_0^1\tr(e^{-tA_s^2}(A_+-A_-))ds=\Big(\frac{t}{\pi}\Big)^{1/2}\int_\bbR e^{-ts^2}\xi(s; A_{+},A_-)ds.
\end{equation*}

Since the functions $s\mapsto e^{-ts}, s\in\bbR, t>0$, is a Schwartz function,
it follows from  Krein's trace formula \eqref{trace_formula_left-hand side} that
$\tr(e^{-t\bsH_2}-e^{-t\bsH_1})=-t\int_0^\infty \xi(\lambda;\bsH_2,\bsH_1) e^{-t\lambda}\, d\lambda.$

Thus,
\begin{align}\label{both_ssf}
\int_0^\infty &\xi(\lambda;\bsH_2,\bsH_1) e^{-t\lambda}\, d\lambda=\Big(\frac1{\pi\cdot t}\Big)^{1/2}\int_\bbR \xi(s;A_+,A_-)e^{-ts^2}ds\nonumber\\
&=\Big(\frac1{\pi\cdot t}\Big)^{1/2}\int_0^\infty \frac{\xi(\sqrt{s};A_+,A_-)+\xi(-\sqrt{s};A_+,A_-)}{\sqrt{s}}e^{-ts}ds,
\end{align}
where for the last integral we used the substitutions $s\mapsto \sqrt{s}$ and $s\mapsto -\sqrt{s}$ for the integrals on $(0,\infty)$ and on $(-\infty, 0)$, respectively.

Let us denote by $L$ the Laplace transform on $L_{loc}^1(\bbR)$. It is well-known that $L(\frac1{\pi \sqrt{s}})(t)=\frac1{\sqrt{\pi \, t}}$ (see e.g. \cite[29.3.4]{AbSt_table}). Therefore, introducing 
$$\xi_0(s):=\frac{\xi(\sqrt{s};A_+,A_-)+\xi(-\sqrt{s};A_+,A_-)}{\sqrt{s}},\, s\in[0,\infty),$$ equality \eqref{both_ssf} can be rewritten as 
$L\big( \xi(\lambda;\bsH_2,\bsH_1) \big)(t)=L\big(\frac1{\pi \sqrt{s}}\big)(t) \cdot L\big(\xi_0(s)\big)(t).$
By \cite[Proposition 1.6.4]{ABHNbook_Laplace} the right-hand side of the previous equality is equal to 
$L\Big(\frac1{\pi \sqrt{s}}\ast \xi_0(s)\Big)(t).$
Therefore, by the uniqueness theorem for the Laplace transform (see e.g. \cite[Theorem 1.7.3]{ABHNbook_Laplace}) we have 
$\xi(\lambda;\bsH_2,\bsH_1)=\big(\frac1{\pi \sqrt{s}}\ast \xi_0(s)\big)(\lambda)$
for a.e. $\lambda\in[0,\infty)$. 
Thus, for a.e. $\lambda\in[0,\infty)$ we have 
\begin{align*}
\xi(\lambda;\bsH_2,\bsH_1)&=\frac1\pi \int_0^\lambda \frac{1}{\sqrt{\lambda-s}}\xi_0(s)ds\\
&=\frac1\pi \int_0^\lambda \frac{1}{\sqrt{\lambda-s}}\frac{\xi(\sqrt{s};A_+,A_-)+\xi(-\sqrt{s};A_+,A_-)}{\sqrt{s}}ds\\
&=\frac1\pi \int_0^\lambda \frac{\xi(\sqrt{s};A_+,A_-)}{\sqrt{s}\sqrt{\lambda-s}}ds+\frac1\pi \int_0^\lambda \frac{\xi(-\sqrt{s};A_+,A_-)}{\sqrt{s}\sqrt{\lambda-s}}ds\\
&=\frac1\pi \int_{-\sqrt{\lambda}}^{\sqrt{\lambda}} \frac{\xi(s; A_+,A_-)ds}{\sqrt{\lambda-s^2}},
\end{align*}
as required.
\end{proof}


Having established Pushnitski's formula, we can now prove Theorem \ref{thm_WI=ssf}, which provides a relation between the Witten index of the operator $\bsD_\bsA^{}$  and the spectral shift function $\xi(\cdot;A_+,A_-)$.
We start with definition of the Witten index. 

\begin{definition} \lb{d8.1} 
Let $T$ be a closed, linear, densely defined operator acting in $\cH$ and  
suppose that for some $t_0 > 0$
$e^{- t_0 T^* T} - e^{- t_0 TT^*}\in \cL_1(\cH).  $  
Then $\big(e^{-t T^*T} - e^{-t TT^*}\big) \in \cL_1(\cH)$ for all $t >t_0$, and one introduces the  (semigroup) regularized Witten index $W_s (T)$ of $T$ 
\begin{equation*}
W_s(T) = \lim_{t \uparrow \infty} \tr_{\cH}\big(e^{-t T^*T} - e^{-t TT^*}\big),    
\end{equation*}
whenever the limit exists.
\end{definition} 

Recall that for a closed densely defined operator the Witten index has also a different regularisation by resolvent differences (see \cite{BGGSS87}, \cite{CGPST_Witten}). Namely, suppose that for some $($and hence for all\,$)$
$z \in \bbC \backslash [0,\infty)$, we have that 
\begin{equation*}
\big[(T^* T - z)^{-1} - (TT^* - z)^{-1}\big] \in \cL_1(\cH).  
\end{equation*}
Then the resolvent regularized Witten index $W_r (T)$ of $T$ is defined by
\begin{equation*}
W_r(T) = \lim_{\lambda \uparrow 0} (- \lambda) \tr\big((T^* T - \lambda )^{-1}
- (T T^* - \lambda )^{-1}\big)
\end{equation*}
whenever this limit exists.

However, in our setting we aim to consider the case when $T$ is a differential type operator. In this case, it is  typical that the difference of resolvents $(T^* T - z)^{-1} - (TT^* - z)^{-1}$ belongs to a higher  Schatten class as would be expected for the study of differential operators in higher dimensions. Therefore, we need to modify the definition of resolvent regularisation of the Witten index.

\begin{definition}Let $T$ be a closed, linear, densely defined operator acting in $\cH$ and let $k\in \bbN$. Suppose that for all $\lambda<0$ we have that 
\begin{equation*}
\big[(T^* T - \lambda)^{-k} - (TT^* - \lambda)^{-k}\big] \in \cL_1(\cH).  
\end{equation*}
Then the $k$-th resolvent regularized Witten index $W_{k,r} (T)$ of $T$ is defined by
\begin{equation*}
W_{k,r}(T) = \lim_{\lambda \uparrow 0} (- \lambda)^k \tr\big((T^* T - \lambda )^{-k}
- (T T^* - \lambda )^{-k}\big)
\end{equation*}
whenever this limit exists.
\end{definition}

\begin{remark}We note that the $k$-th resolvent regularised Witten index is the limit (as $\lambda\uparrow 0$) of the so-called homological index (see \cite{CGK15, CGK16}). \hfill $\diamond$
\end{remark}
To relate the Witten index of the operator $\bsD_\bsA^{}$ to the spectral shift function for the pair $(A_+,A_-)$, we first recall some necessary definitions.

\begin{definition} \label{d5.1}
Let $f \in L^1_{loc} (\bbR)$ and $h > 0$. 
\begin{enumerate}
\item The point  $x \in \bbR$ is called a right Lebesgue point of $f$ if
there exists an $\alpha_+ \in \bbC$ such that
\begin{equation*}
\lim_{h \downarrow 0} \f{1}{h} \int_{x}^{x + h}|f(y) - \alpha_+|  dy= 0.   
\end{equation*}
One then denotes $\alpha_+ = \Lf(x_+)$. 

\item The point $x \in \bbR$ is called a left Lebesgue point of $f$ if
there exists an $\alpha_- \in \bbC$ such that
\begin{equation*}
\lim_{h \downarrow 0} \f{1}{h} \int_{x - h}^{x}  |f(y) - \alpha_-| dy= 0.  
\end{equation*}
One then denotes $\alpha_- = \Lf(x_-)$. 

\end{enumerate}
\end{definition}

%

To establish our results for resolvent regularisations of the Witten index we also need the following lemma, which establishes a more general version of \cite[Lemma ~4.1~(ii)]{CGPST_Witten}. Its proof is a verbatim repetition of the argument of  \cite[Lemma 4.1 (ii)]{CGPST_Witten} and is therefore omitted.  

\begin{lemma} \label{lem_limit_is_Lvalue} Let $k\in\bbN$. 
Introduce the linear operator 
$$T_k: 
L^1\big((0,\infty);(\nu+1)^{-k-1}d\nu\big) \to L^1_{loc}((0,\infty); d\lambda)$$
by setting 
$$(T_kf)(\lambda)=-k\lambda^k \int_0^\infty (\nu+\lambda)^{-k-1} f(\nu) d \nu,\quad\lambda > 0.$$
 If $0$ is a Lebesgue point for $f\in L^1\big((0,\infty);(\nu+1)^{-k-1}d\nu\big)$, then
\begin{equation*} 
\lim_{\lambda\downarrow0}(T_k f)(\lambda)= \Lf (0_+).
\end{equation*} 
\end{lemma}

We are now in a position to state our main result in this present Section.
Its proof closely follows the argument used in  \cite{CGPST_Witten}.

\begin{theorem} \label{thm_WI=ssf}
Assume Hypothesis \ref{hyp_final} and assume that $0$ is a right 
and a left Lebesgue point of $\xi(\cdot; A_+, A_-)$.
 Then $0$ is a right Lebesgue point of 
$\xi(\cdot ; \bsH_2, \bsH_1)$ 
\begin{equation*}
\Lxi(0_+; \bsH_2, \bsH_1) 
= [\Lxi(0_+; A_+,A_-) + \Lxi(0_-; A_+, A_-)]/2   
\end{equation*}
and the Witten indices $W_s(\bsD_\bsA^{})$ and $W_{k,r}(\bsD_\bsA^{})$, $k\geq m$, exist and equal 
\begin{align*} 
\begin{split}
W_s(\bsD_\bsA^{}) &= W_{k,r}(\bsD_\bsA^{})=\Lxi(0_+; \bsH_2, \bsH_1) \\
&= [\Lxi(0_+; A_+,A_-) + \Lxi(0_-; A_+, A_-)]/2
\end{split}
\end{align*}
\end{theorem}
\begin{proof}
First, one rewrites \eqref{eq_Push} in the form,
\begin{equation} 
\xi(\lambda; \bsH_2, \bsH_1) = \frac{1}{\pi}\int_0^{\lambda^{1/2}}
\frac{d \nu \, [\xi(\nu; A_+,A_-) + \xi(-\nu; A_+,A_-)]}{(\lambda-\nu^2)^{1/2}},  
\quad \lambda > 0.     \label{55f}
\end{equation} 
Define the function $f(\nu) = [\xi(\nu;A_+,A_-)+\xi(-\nu;A_+,A_-)]$. By assumption, $0$ is a right 
and a left Lebesgue point of $\xi(\,\cdot\,\, ; A_+, A_-)$, and therefore, $0$ is a right Lebesgue point of $f$.  Equality \eqref{55f} together with \cite[Lemma 4.1 (i)]{CGPST_Witten}
implies that 
%
 $0$ is a right Lebesgue point of 
$\xi(\,\cdot\,\, ; \bsH_2, \bsH_1)$ and 
$$\Lxi(0_+; \bsH_2, \bsH_1) = \frac12\Lf(0_+)=\frac12(\Lxi(0_+; A_+,A_-) + \Lxi(0_-; A_+, A_-)).$$

Next, to prove the equality for the semigroup regularised Witten index $W_s(\bsD_\bsA^{})$ we introduce the function
$\Xi(r; \bsH_2,\bsH_1)=\int_{0}^{r}\xi(s;\bsH_2, \bsH_1)\,ds, \quad r >0.$
By Krein's trace formula \eqref{trace_formula_left-hand side} we have that
\begin{align*}
\begin{split}
\frac{1}{t} {\tr}\big(e^{- t \bsH_2} - e^{- t \bsH_1}\big)
= - \int_{0}^{\infty} \xi (s; \bsH_2, \bsH_1)\, e^{-ts}\, ds \\
= - \int_{0}^{\infty}\, e^{-ts} d \, \Xi (s; \bsH_2,\bsH_1).
\end{split}
\end{align*}
We have already established, that $0$ is a right Lebesgue point of $\xi (\, \cdot \,; \bsH_2, \bsH_1)$. Hence, one obtains that
\begin{equation*}
\lim_{r \downarrow 0+} \frac{\Xi (r; \bsH_2,\bsH_1)}{r} = \Xi'(0_+; \bsH_2,\bsH_1) = \Lxi(0_+; \bsH_2,\bsH_1)
\end{equation*}
exists. Then, an Abelian theorem for Laplace transforms \cite[Theorem\ 1, p.\ 181]{Wi41} (with $\gamma=1$) implies that
\begin{align*}
-\lim_{t \to \infty} {\tr}_{\cH} \big(e^{- t\bsH_2} - e^{- t \bsH_1}\big)= \lim_{r \downarrow 0+} \frac{\Xi (r; \bsH_2,\bsH_1)}{r}=\Lxi(0_+; \bsH_2,\bsH_1).
\end{align*}

To prove the equality for the $k$-th resolvent regularisation $W_{k,r}(\bsD_\bsA^{})$ we write 
\begin{align*}
W_{k,r}(\bsD_\bsA^{})&=\lim_{\lambda\uparrow 0} (-\lambda)^k \tr\big( (\bsH_2-\lambda)^{-k}-(\bsH_1-\lambda)^{-k}\big)\\
&=-k\lim_{\lambda\uparrow 0} \int_0^\infty (\nu-\lambda)^{-k-1}\xi(\nu; \bsH_2, \bsH_1)\,d\nu\\
&=\lim_{\lambda\uparrow 0} \big(T_k \xi(\cdot; \bsH_2,\bsH_1)\big)(-\lambda),
\end{align*}
where $T_k$ is the operator introduced in Lemma \ref{lem_limit_is_Lvalue}. Since $0$ is a right Lebesgue point for $\xi(\cdot; \bsH_2,\bsH_1)\in L^1\big((0,\infty);(\nu+1)^{-k-1}d\nu\big)$, $k\geq m$, Lemma~\ref{lem_limit_is_Lvalue} implies the required result:
$$W_{k,r}(\bsD_\bsA^{})=\Lxi(0_+; \bsH_2, \bsH_1)=\frac12(\Lxi(0_+; A_+,A_-) + \Lxi(0_-; A_+, A_-)).$$
\end{proof}

\section{Example of the Dirac operator in $\bbR^d$}\label{ch_examples}

In this section we supplement the abstract discussion by an example for which our general assumption of Hypothesis \ref{hyp_final} holds. Our primary example is the multidimensional Dirac operator and its perturbations given by multiplication operators by matrix valued  functions. Thus, our framework is indeed suitable
for differential operators on certain non-compact manifolds.
%
%
%

%
Throughout this section we fix $d\in\mathbb{N}.$ 
For each $k=1,\ldots,d$, we denote by $\partial_k$ the operators of partial differentiation, that is operators in $L^2(\R^d)$ defined as 
$$\partial_k=-i\frac{\partial}{\partial t_k},\quad \dom(\partial_k)=W^{1,2}(\R^d).$$
We denote the tuple $(\partial_1,\dots, \partial_d)$ by $\nabla.$

 Let $n(d)=2^{\lceil\frac{d}{2}\rceil}$. Let $\gamma_k\in M_{n(d)}(\mathbb{C}),$ $0\leq k\leq d,$ be Clifford algebra generators, that is,
\begin{enumerate}
\item $\gamma_k=\gamma_k^*$ and $\gamma_k^2=1$ for $0\leq k\leq d.$
\item $\gamma_{k_1}\gamma_{k_2}=-\gamma_{k_2}\gamma_{k_1}$ for $0\leq k_1,k_2\leq d,$ such that $k_1\neq k_2.$
\end{enumerate}
We use the notation $\gamma=(\gamma_1,\dots, \gamma_d)$.

\begin{definition}Let $m\geq 0$.
Define the Dirac operator as an unbounded operator $\cD$ acting in the Hilbert space $\mathbb{C}^{n(d)}\otimes L^2(\mathbb{R}^d)$ with domain $\dom(\cD)=\C^{n(d)}\otimes W^{1,2}(\R^d)$ by the formula
\begin{equation}\label{def_D}
\cD=\gamma\cdot \nabla+m\gamma_0.
\end{equation}
\end{definition}

It is well-known that the operator $\cD$ is self-adjoint. Furthermore, 
$\cD^2=-\Delta+m^2,$
where $-\Delta=-\sum_{k=1}^d\frac{\partial^2}{\partial^2 x_k}$  is the Laplace operator.

Suppose that $V=\{\phi_{ij}\}_{i,j=1}^{n(d)}$ is { a hermitian matrix} of functions, such that $\phi_{ij}\in L^\infty(\R^d)$. 
We also write $V$ for the bounded operator of multiplication by this matrix function $V$ on the Hilbert space $\mathbb{C}^{n(d)}\otimes L^2(\mathbb{R}^d)$. We also identify a function $f\in L_\infty(\bbR^d)$ with the operator on $L_2(\bbR^d)$ of multiplication by $f$.

For the example of the current section  we set
$A_-=\cD, \quad B_+=V$
and 
$B(t)=\theta(t) B_+,$
where $\theta$ satisfies \eqref{theta}.
%
Recall, (see \cite{BS_estimates} and \cite{Simon_book}) that the space $l_1(L_2)(\bbR^d)$ is defined as  
$l_1(L^2)(\bbR^d):=\bigg\{f\in L^0(\bbR^d)\,:\,\sum_{n\in\bbZ^d}\|f\chi_{Q+n}\|_{2}<\infty\bigg\},$
where $Q$ denotes the unit cube in $\bbR^d$ centered at $0$. The space $l_1(L_2)(\bbR^d)$ is equipped with the norm defined by 
\begin{equation*}
\|f\|_{l_1(L^2)(\bbR^d)}:=\sum_{n\in\bbZ^d}\|f\chi_{Q+n}\|_{2},\qquad f\in l_1(L^2)(\bbR^d).
\end{equation*}

\begin{proposition}\label{Dirac_hyp_iii}
 Assume that $V=\{\phi_{ij}\}_{i,j=1}^{n(d)}$ is such that 
$\phi_{ij}\in l_1(L^2)(\mathbb{R}^d)$. Then $V$ is a $d$-relative trace-class perturbation with respect to $\cD$, that is Hypothesis \ref{hyp_final_special} (ii) is satisfied. 
\end{proposition}
\begin{proof}
Since $\phi_{ij}\in l_1(L^2)(\bbR^d)$, \cite[Theorem 4.4]{Simon_book} implies that 
the matrix elements $\phi_{ij}{(-\Delta+1)^{-\frac{d+1}{2}}}$ are trace class operators (on $L^2(\bbR^d)$) for all $1\leq i,j\leq n(d)$. Hence, the operator
$$V{(\cD^2+1)^{\frac{-d-1}{2}}}=\Big({\phi_{ij}}{(-\Delta+1)^{\frac{-d-1}{2}}}\Big)_{i,j=1}^{n(d)}$$
is a trace class operator (on $\mathbb{C}^{n(d)}\otimes L^2(\mathbb{R}^d)$). Therefore, $V (\cD+i)^{-1-d}\in \cL_1(\mathbb{C}^{n(d)}\otimes L^2(\mathbb{R}^d))$. 
\end{proof}

Our next task is to  establish a sufficient condition on the matrix $V$ for Hypothesis \ref{hyp_final_special} (iii) to hold.
%
%
Since $\cD^2=-\Delta+m^2$, we have that
$$[\cD^2,V]=\Big([-\Delta, {\phi_{ij}}]\Big)_{i,j=1}^{n(d)},$$
whenever the commutators $[-\Delta, {\phi_{ij}}]$ are well-defined. Therefore, introducing $L_{-\Delta}^j, j\in\bbN,$ as in \eqref{def_L} (with $A^2=-\Delta$), we obtain that 
$V\in\dom(L_{\cD^2}^k)$ for some $k\in\bbN$, provided that ${\phi_{ij}}\in\dom(L_{-\Delta}^k)$ for any $i,j=1,\dots, n(d).$ In this case, 
\begin{equation*}
L_{\cD^2}^k(V)=\overline{(1+\cD^2)^{-k/2}[\cD^2,T]^{(k)}}=\Big(L_{-\Delta+m^2}^k({\phi_{ij}})\Big)_{i,j=1}^{n(d)}.
\end{equation*}


\begin{proposition}\label{dom_delta_Laplace}Let $k\in\bbN$ be fixed. Assume that $\phi\in W^{2k,\infty}(\bbR^n).$ Then $\phi\in\bigcap_{j=1}^k\dom(L_{-\Delta+m^2}^j).$
\end{proposition}
\begin{proof}Let $k\in\bbN$ be fixed. 
Since $\phi\in W^{2k,\infty}(\bbR^n)$, we have that $(\Delta)^{j}(\phi\xi)\in L^2(\bbR^n)$ for every $\xi\in\dom(\Delta)^{j}$, $j=1,\dots k$. That is $\phi \dom(\Delta)^{j}\subset \dom(\Delta)^{j}$ for all $j=1,\dots, 2k$.

Recall that $\partial_k=\frac{\partial}{i\partial x_j}$ and if $\phi\in L_\infty(\bbR^d)$ with $\frac{\partial \phi}{\partial x_k}\in L_\infty(\bbR^d), \, k=1,\dots, d$, then 
$\phi\dom(\partial_k)\subset\dom(\partial_k)$ and for all $\xi\in\dom(\partial_k)$ we have  
\begin{equation}
\label{commut_with_D_k}
[\partial_k,\phi]\xi=\frac{1}{i}\frac{\partial\phi}{\partial x_k}\xi,\quad k=1,\dots,d.
\end{equation}

 By \eqref{commut_with_D_k} we have 
 \begin{align*}
[\Delta,\phi]&=\sum_{j=1}^n[\partial_j^2,\phi]=\sum_{j=1}^n\partial_j[\partial_j,\phi]+\sum_{j=1}^n[\partial_j,\phi]\partial_j\\
&=\frac1{i}\sum_{j=1}^n \Big(\partial_j{\frac{\partial \phi}{\partial x_j}}+{\frac{\partial\phi}{\partial x_j}}\partial_j\Big)=\frac1{i}\sum_{j=1}^n \Big(2\partial_j{\frac{\partial \phi}{\partial x_j}}-[\partial_j,{\frac{\partial \phi}{\partial x_j}}]\Big)\\
&=\frac2{i}\sum_{j=1}^n \partial_j{\frac{\partial \phi}{\partial x_j}}+\sum_{j,\ell=1}^n{\frac{\partial^2 \phi}{\partial x_j\partial x_\ell}}.
\end{align*}
Therefore,
\begin{align*}
{(1+m^2-\Delta)^{-1/2}}[\Delta,\phi]&=\frac2{i}\sum_{j=1}^n{\partial_j}{(1+m^2-\Delta)^{-1/2}}{\frac{\partial \phi}{\partial x_j}}\\
&\quad+\sum_{j,\ell=1}^n{(1+m^2-\Delta)^{-1/2}}{\frac{\partial^2 \phi}{\partial x_j\partial x_\ell}}.
\end{align*}

Since $\phi\in W^{2k,\infty}(\bbR^n)$, it follows that the operators ${\frac{\partial^2 \phi}{\partial x_j\partial x_\ell}}$ and $M_{\frac{\partial \phi}{\partial x_j}}, j,\ell=1,\dots,n,$ are bounded. Since the operator ${\partial_j}{(1+m^2-\Delta)^{-1/2}}$ is also bounded, we infer that 
$$\overline{{(1+m^2-\Delta)^{-1/2}}[\Delta,\phi]}\in\cB(L^2(\bbR^n)).$$

Continuing this process, we obtain that 
\begin{equation}\label{inc_L_laplace}
\overline{{(1-\Delta)^{-j}}[\Delta,\phi]^{(j)}}\in\cB(L^2(\bbR^n)),\quad j=1,\dots,k,
\end{equation}
that is $\phi\in\bigcap_{j=1}^{2k}\dom(L_{-\Delta}^j).$
\end{proof}

Combining now Propositions \ref{Dirac_hyp_iii} and \ref{dom_delta_Laplace} we arrive at the following 

\begin{theorem}\label{Dirac_hyp_final}
Let $A_-=\cD$ be the Dirac operator on $C^{n(d)}\otimes L^2(\bbR^d)$ defined by \eqref{def_D}, \, $d\in\bbN$. 
 Assume that $V=\{\phi_{ij}\}_{i,j=1}^{n(d)}$ is such that 
$$\phi_{ij}\in l_1(L^2)(\bbR^d)\cap W^{4p,\infty}(\bbR^d),\quad i,j=1,\dots, n(d).$$
Then the operator $A_-=D^{(d)}$ and the perturbation $B_+=V$ satisfy Hypothesis \ref{hyp_final_special} (and hence also Hypothesis \ref{hyp_final}) with $p=d$.
\end{theorem}
%
Everywhere below we assume that the  perturbation $V=\{\phi_{ij}\}_{i,j=1}^{n(d)}$ satisfies the assumption of Theorem \ref{Dirac_hyp_final}.
As $\cD^2=-\Delta+m^2$, it follows that the operator $\cD$ has purely absolutely continuous spectrum, which coincides with $(-\infty -m]\cup [m,\infty)$. 
In the case, when $m$ is strictly positive, { the assumption that the operator $V(\cD+i)^{-d-1}$ is compact together with Weyl's theorem guarantees that } the  operator $\cD+V$ has purely discrete spectrum in the interval $(-m,m)$. In particular, if $\cD+V$ is also invertible, then by Theorem \ref{t8.iff} the 
 corresponding operator
 \begin{equation}\label{eq_D_A_in_Dirac}
 \bsD_\bsA=\frac{d}{dt}\otimes 1+1\otimes \cD+ \theta V,
 \end{equation}
(see \eqref{D_A}) is Fredholm. Furthermore, since in this case the spectral shift function $\xi(\cdot; \cD+V,\cD)$ for the pair $(\cD+V, \cD)$ is constant in a neighbourhood of zero, Theorem \ref{thm_WI=ssf} implies that 
$\iindex(\bsD_\bsA)=\xi(0; \cD+V,\cD).$
If the operator $\cD+V$ is not invertible, then the operator $\bsD_\bsA$ is no longer Fredholm. However, since in this case the spectral shift function $\xi(\cdot; \cD+V,\cD)$ is left and right continuous at zero, it follows that $0$ is, in particular, left and right Lebesgue point of $\xi(\cdot;\cD+V,\cD)$. Hence, Theorem \ref{thm_WI=ssf} again implies that 
\begin{equation}
\label{eq_WI_ssf_Dirac_massive}
W_s(\bsD_\bsA)=\frac12[\xi(0+; \cD+V,\cD)+\xi(0-; \cD+V,\cD)].
\end{equation}

In particular, if $d=3$ and the potential $V$ is a magnetic potential, that is 
$$V=\sum_{n=1}^3\gamma_jA_j,\quad  A_j\in L_3(\bbR^3)\cap L_1(\bbR^3),$$
then \cite[Section 5.1]{Safronov} implies that 
$\xi(0; \cD+V,\cD)=0$, and therefore, by \eqref{eq_WI_ssf_Dirac_massive} we obtain that 
$W_s(\bsD_\bsA)=0.$

Our main interest lies in the massless Dirac operator, where $m=0$. In this case, the spectrum of $\cD$ covers the whole real line, and therefore, whatever the potential $V$, the operator $\bsD_\bsA$ in \eqref{eq_D_A_in_Dirac} is never Fredholm. 
However, Theorem \ref{Dirac_hyp_final} guarantees that the pair 
$A_-=\cD,\quad B_+=V,$ both
satisfy Hypothesis \ref{hyp_final}, and therefore, by Theorem \ref{thm_WI=ssf} to study whether the Witten index of $\bsD_\bsA$ exists in this case it is sufficient to study the spectral shift function $\xi(\cdot; \cD+V,\cD)$ for the pair $(\cD+V, \cD)$ and its behaviour near zero. 

Although the spectral shift function for the second order operators is well studied, there is a sparse literature available for the spectral shift function for the first order operators $\cD$ and $\cD+V$ (see e.g. \cite{Safronov, Tiedra_2011, Bruneau_Robert_99}). Furthermore, the majority of papers on this topic study the massive case, $m>0$, and are, therefore, not applicable to our case. Our initial investigation, in collaboration with F. Gesztesy and R. Nichols, shows that for a sufficiently good perturbation $V$, the spectral shift function $\xi(\cdot; \cD+V,\cD)$ is left and right continuous at zero even in the massless case. Hence, the Witten index of the operator $\bsD_\bsA$ exists and 
$W_s(\bsD_\bsA)=\xi(0; \cD+V,\cD).$
However, the proof of left/right continuity of the spectral shift function $\xi(\cdot; \cD+V,\cD)$ at zero involves a generalisation of the original formula for the spectral shift function due to Krein (see \cite{Yafaev_general}) in terms of a perturbation determinant { as well as to an extensive analysis of the spectrum of the perturbed operator $\cD+V$. } The full details of this approach as well as the exact statement of continuity of $\xi(\cdot; \cD+V,\cD)$ involve lengthy arguments
and we defer the discussion to a separate manuscript \cite{CGLNSZ_ssf_cont_zero}.
%
%
%
%
%
%
%
%

\label{Bibliography}
 
\bibliographystyle{amsalpha} 
\bibliography{Bibliography}
\end{document}